\documentclass[12pt,reqno]{amsart}
\usepackage{cases}
\usepackage{amscd}
\usepackage{amsfonts}
\usepackage{amssymb}
\usepackage{amsmath}
\usepackage{threeparttable}

\usepackage{graphicx}
\usepackage{subfigure}
\usepackage{amsthm}
\usepackage {stmaryrd} 
\usepackage{algorithm} %format of the algorithm
\usepackage{algorithmic} %format of the algorithm

\usepackage{cleveref}
\usepackage{color}
\usepackage{datetime}

\allowdisplaybreaks

\makeatletter
\newcommand\figcaption{\def\@captype{figure}\caption}
\newcommand\tabcaption{\def\@captype{table}\caption}
\makeatother

\usepackage{geometry}
\usepackage{algorithm}
\usepackage{multirow}
\renewcommand{\theequation}{\thesection.\arabic{equation}}
\def\bq{\begin{equation}}
	\def\eq{\end{equation}}
\def\br{\begin{eqnarray}}
	\def\er{\end{eqnarray}}
\def\brr{\bq\begin{array}{rlll}}
	\def\err{\end{array}\eq}

\def\pmb#1{\mbox{\boldmath $#1$}}\def\text#1{\hbox{#1}}

\newtheorem{theorem}{Theorem}[section]

\newtheorem{lemma}{Lemma}[section]

\newtheorem{rem}{Remark}[section]

\newcommand{\bsub}{\begin{subequations}}
	\newcommand{\esub}{\end{subequations}$\!$}

\newtheorem{example}{\textbf{Example}}[section]
\numberwithin{equation}{section}

\newcommand{\BX}{\mathbf{B}(X)}

\begin{document}
	
\title[]{Enhanced gradient recovery-based a posteriori error estimator and adaptive finite element method for elliptic equations}

\author[Y. Liu, J. Xiao, N. Yi  and H. Cao]{Ying Liu$^\dagger$, Jingjing Xiao$^{\ddagger, *}$, Nianyu Yi$^\S$ and Huihui Cao$^\S$ }

\address{$^\dagger$ Department of Mathematics, School of Science Xi’an University of Technology, Xi’an, 710000, China}
\email{liuyinglixueyuan@xaut.edu.cn}
\address{$^\ddagger$ School of Mathematics, Shandong University, Jinan, 250100, China}
\email{jingjing\_xiao@mail.sdu.edu.cn}
\address{$^\S$ School of Mathematics and Computational Science, Xiangtan University; Hunan Key Laboratory for Computation and Simulation in Science and Engineering; Hunan National Center for Applied Mathematics, Xiangtan, 411105, China}
\email{yinianyu@xtu.edu.cn (N. Yi);\ 201721511145@smail.xtu.edu.cn (H. Cao)}

\date{\today}

\begin{abstract}
 Recovery type a posteriori error estimators are popular, particularly in the engineering community, for their computationally inexpensive, easy to implement, and generally asymptotically exactness.
Unlike the residual type error estimators, one can not establish upper and lower a posteriori error bounds for the classical recovery type error estimators without the saturation assumption. 
In this paper, we first present three examples to show the unsatisfactory performance in the practice of standard residual or recovery-type error estimators, then, an improved gradient recovery-based a posteriori error estimator is constructed. 
The proposed error estimator contains two parts, one is the difference between the direct and post-processed gradient approximations, and the other is the residual of the recovered gradient.
The reliability and efficiency of the enhanced estimator are derived.
Based on the improved recovery-based error estimator and the newest-vertex bisection refinement method with a tailored mark strategy, an adaptive finite element algorithm is designed. 
We then prove the convergence of the adaptive method by establishing the contraction of gradient error plus oscillation.
 Numerical experiments are provided to illustrate the asymptotic exactness of the new recovery-based a posteriori error estimator and the high efficiency of the corresponding adaptive algorithm.	
 \end{abstract} 

\maketitle
	
%\tableofcontents
\section{Introduction}\label{sec1}
Adaptive finite element methods \cite{br1978, mn,  Bonito2024AdaptiveFE} are widely used for numerically solving partial differential equations, especially for solutions with singularity or multiscale properties.  
Based on the principle of uniform distribution of errors, the adaptive algorithm adjusts the mesh such that the errors are ``equally'' distributed over the computational mesh.
A posteriori error estimation
\cite{VER1996, Ainsworth2000APE, ibts2001}, which provides information about the size and the distribution of the error, is an essential ingredient of adaptive finite element methods.
Error estimators are computable quantities depending on finite element approximations and known data that locate accurate sources of global and
local error. In global, we use the a posteriori error estimator as the stop criterion to determine whether the finite element solution is an acceptable approximation.
Locally, we use the a posteriori error estimator as an indicator which shows the error distribution and guides the local mesh adaption.

There is a large numerical analysis literature on adaptive finite element methods \cite{marking1996, Cascn2008QuasiOptimalCR, Kreuzer2011DecayRO, CCMM2014}, and various
kinds of a posteriori estimates have been proposed for different problems \cite{Ainsworth2000APE, Durn1991OnTA, Triki2024APE}. 
Basically, there are two
types of a posteriori error estimations, one is residual type error estimation, and the other is the
recovery type error estimation.
Residual type error estimations, originally introduced by Babu{\v{s}}ka and Rheinboldt \cite{br}, to estimate errors, consider local residuals of the numerical solution. 
Today, residual type a posteriori error estimates are well studied for a large class of elliptic model problems \cite{cn,bv,mn1}. 
It is proved that the residual estimator provides upper and lower bounds of the actual error in a suitable norm.
However, in practical simulation, the error estimators are usually larger than the actual unknown error. Hence, it overestimates the true error and causes over-refinement of the mesh.
Recovery type a posteriori error estimations, which adopt a certain norm of the difference between the direct and post-processed approximations of the gradient 
(or other quantities) as an indicator, have gained wide popularity since the work of Zienkiewicz and Zhu \cite{zz}. Estimators of the recovery type possess many attractive features. In particular, their ease of implementation, generality, and ability to produce quite accurate estimators have led to their widespread adoption. Recovery type a posteriori error estimator is based on a suitable finite element post-processing technique including the gradient recovery \cite{dy, fa,yz, za,Deng2022RECOVERYBASEDAP}, flux recovery \cite{cz,cz1} or functional recovery \cite{functional2000}. 
Gradient recovery is a post-processing technique that reconstructs improved gradient approximations from finite element solutions, which is widely used in engineering practice for its superconvergence of recovered derivatives and its robustness as an a posteriori error estimator. 
Now, different kinds of post-processing techniques are developed based on weighted averaging \cite{hy},
 the local least-squares methods \cite{zz,hy1,ppr2004}, the local or global projections, post-processing interpolation, and so on. 
 Under the assumption that the recovery operator is superconvergent \cite{VER1996}, the corresponding error estimator is asymptotically exact \cite{Durn1991OnTA}. 
Unfortunately, without the saturation assumption that the recovery provides a
better approximation than the numerical approximation does, one can not derive the reliability and efficiency of the recovery type estimator.
In some cases, such as the diffusion problem with  discontinuous coefficients of the load function is orthogonal to the finite element space \cite{mn}, the recovery process may not provide improved gradient approximation, then the corresponding error estimator is not reliable and can lead to adaptive refinement
 completely failing to reduce the global error, which in turn produces a wrong finite element approximation. 
The purpose of this paper is to derive a gradient recovery-based error estimator with guaranteed upper and lower bounds of the actual error.
In view of some interesting numerical findings, an improved error estimator, which contains an additional term that ensures reliability, is proposed and proved to be reliable and efficient. We formulate an adaptive algorithm that is driven by this estimator and a tailored mark strategy and establish its convergence theory.

The rest of this paper is organized as follows. In Section \ref{sec2}, we give a description of the linear elliptic model problem and its finite element method. 
The poor performance of the classical error estimators, i.e., residual estimators are not asymptotically accurate, the gradient recovery method may not provide a better approximation, and gradient recovery-based error estimator may lead to over-refinement, are explained and demonstrated by three examples in Section \ref{sec3}. Then an improved gradient recovery-based error estimator is presented and analyzed. In Section \ref{sec4}, we present the adaptive finite element algorithm and prove its convergence result. Several numerical experiments are reported in Section \ref{sec5} to demonstrate the performance of our improved estimator.

\section{Model problem and its finite element scheme}\label{sec2}
Let $\Omega \subset R^d (d=2, 3)$ be an open and bounded polyhedral domain. 
Following the standard notation,  we use $L^2(\Omega)$ to denote the space of all square-integrable functions and its norm is denoted by $\|\cdot\|$. Let $H^s(\Omega)$ be the standard Sobolev space with norm $\|\cdot\|_s$ and seminorm $|\cdot|_s$. Furthermore, $\|\cdot\|_{s,D}$ and $|\cdot |_{s,D}$ denote the norm $\|\cdot\|_s$ and the semi-norm $|\cdot |_s$ restricted to the domain $D\subset\Omega$, respectively. $W^{k,\infty}(\Omega)$ denotes the standard Sobolev space with norm $\|\cdot\|_{k,\infty, \Omega}$ and seminorm $|\cdot|_{k,\infty, \Omega}$. We also use the notation $H^1_0(\Omega)$ for the functions that belong to $H^1(\Omega)$ and their trace vanishes on $\partial \Omega$.  Moreover, we shall use $C$, with or without subscript, for a generic constant
independent of the mesh size and it may take a different value at each occurrence.

We consider the adaptive finite element method for the following problem
\begin{align}
\left\{
\begin{aligned}
-\nabla \cdot (A\nabla u) = f \qquad & in\quad \Omega, \\
u = 0 \qquad &on \quad \partial \Omega,
\end{aligned}
\right.\label{Model problem}
\end{align}
where $f \in L^2(\Omega)$, $A$ is a piecewise constant positive definite symmetric matrix, namely, we assume that there exists a partition of the domain $\Omega$ into a finite set of Lipschitz polygonal domains $\{\Omega_j\}_{j = 1}^N,\, N\ge1$, such that $A$ is constant on each $\Omega_j$.

The corresponding variational form of the model problem \eqref{Model problem} is to find $u\in H_{0}^1(\Omega)$ such that
\begin{equation}\label{yuan}
	a(u,v)=(f,v),\quad\forall v\in H_{0}^1(\Omega),%\eqno(2.8)
\end{equation}
where the bilinear and the linear forms are defined by
$$a(u,v)=\int_\Omega(A \nabla u)(\nabla v){\rm d}x,\quad(f,v)=\int_\Omega fv {\rm d}x,$$
respectively. The energy norm is denoted by
$$\interleave v\interleave_{D}=\|A ^{\frac{1}{2}}\nabla v\|_{0,D}.$$

Let $\mathcal{T}_h=\{K\}$ be a conforming triangulation of domain $\Omega$. For each element $K$, $h_K$ is the diameter of $K$, and $h=\max_{K\in \mathcal{T}_h}h_K$ is the mesh size of $\mathcal{T}_h$.
Assume that the triangulation $\mathcal{T}_h$ is regular; i.e., for all $K\in \mathcal{T}_h$, there exists a positive constant $\kappa$ satisfies 
$h_K\leq \kappa\rho_K$, with $\rho_K$ denoting the diameter of the largest circle that may be inscribed in $K$. 
If the number of sub-domain satisfies $N>1$, let $\mathcal{T}_h^j $ be a triangulation of $\Omega_j$ and assume that interfaces $I = \{\partial \Omega_i \cap \partial \Omega_j:\, i,j = 1,\cdots N\}$ do not cut through any element $K\in \mathcal{T}_h$, so that the discontinuities of coefficient occur only across mesh edge.
Define the set of all edges of the triangulation by
\[\mathcal{E}:=\mathcal{E}_\Omega\cup\mathcal{E}_{\partial \Omega},\]
where $\mathcal{E}_\Omega$ is the set of all interior element edges, $\mathcal{E}_{\partial \Omega}$ is the set of boundary edges. For the edge $e\in \mathcal{E}$, let $\textbf{n}_e$ be the unit vector normal to $e$. When $e\in \mathcal{E}_\Omega$,  $e=K_e^+\cap K_e^-$ is the common edge shared by $K_e^+$ and $K_e^-$.
Let $V_h$ be the continuous piecewise linear finite element space,
\[V_h=\{v\in H^1(\Omega):v|_K\in P_1(K),\quad\forall K\in \mathcal{T}_h\}.\]
Set \[V_h^0=\{v\in V_h:v=0 ~\text{on} ~{\partial \Omega} \},\]
then the corresponding finite element discrete scheme of \eqref{yuan} is to find $u_h\in V_h^0 $ such that
\begin{equation}\label{lisan}
	a(u_h,v_h)=(f,v_h),\quad \forall\,v_h\,\in \,V_h^0.
\end{equation}

\section{Improved gradient recovery-based a posteriori error estimation}\label{sec3}
In this section, we first recall the classical a posteriori error estimators and numerically investigate their performance in the adaptive method for model problem \eqref{Model problem}. 
We shall report three examples to demonstrate that the performance of error estimators can be arbitrarily bad if no care is taken to avoid these difficulties, 
and explain why these estimators fail to guide the adaptive refinement accurately.
Based on this observation, we then present our improved gradient recovery-based error estimator and establish its reliability and efficiency.

\subsection{Poor performances of classical error estimators}
We now recall the residual-type a posteriori error estimators for \eqref{yuan} and \eqref{lisan}. 
The standard residual type estimator on element $K$ is defined as 
\begin{equation}\label{res}
\eta_{res,K}^2=h_K^2\|f\|_{0,K}^2+	\sum_{e\in \partial K} h_e\|J_e(A \nabla u_h)\|_{0,e}^2,
\end{equation}
where the jump of the normal component for any $e\in\mathcal{E}_\Omega$ is
$$J_e(A \nabla u_h)=[A \nabla u_h\cdot \textbf{n}_e]=(A \nabla u_h|_{K_e^+}-A \nabla u_h|_{K_e^-})\cdot \textbf{n}_e,$$
and $J_e(A \nabla u_h)=0$ for any $e\in\mathcal{E}\backslash\mathcal{E}_\Omega$. 
The corresponding global residual type estimator is 
\begin{equation}\label{res_global}
\eta_{res}^2=\sum_{K\in \mathcal{T}_h}h_K^2\|f\|_{0,K}^2+	\sum_{e\in \mathcal{E}} h_e\|J_e(A \nabla u_h)\|_{0,e}^2.
\end{equation}
According to \cite{Petzoldt2002}, the estimator is the global upper bound and local lower bound of the exact error, 
\begin{equation}\label{res_estimate}
\begin{aligned}
\interleave u-u_h\interleave^2&\leq C\eta_{res}^2,\\
\eta_{res, K}^2&\leq C\left(\interleave u-u_h\interleave^2_{\omega_K}+osc_h^2(f, \omega_K)\right),
\end{aligned}
\end{equation}
where the oscillation term is 
\[osc_h^2(f, \omega_K)=\sum\limits_{K\in \omega_K}h_K^2\|f-f_K\|_{0,K}^2,\]
with $f_K$ is the cell average of $f$ on element $K$.

When utilizing the above residual estimator as the driver of the adaptive finite element method, the resulting adaptive meshes will yield the optimal convergence rate.
However, the appearance of unknown constant $C$ in the bounds makes the asymptotic result a bit inconvenient and results in the adaptive algorithm not being terminated timely 
in practical application. 
The following benchmark test is presented to show this unsatisfactory performance.

\begin{example}\label{exam3_1}
[Residual estimators are not asymptotically accurate.]
Let L-shaped domain $\Omega=(-1,1)^2\backslash
(0,1)\times (-1,0)$, we consider the Laplace equation with Dirichlet boundary condition given by an exact solution 
\[u=r^{2/3}\sin\left(2\theta/3\right),\qquad r^2=x^2+y^2,\quad \tan \theta=\frac{x}{y}.\] 
The solution shows singularity at the reentrant corner $(0, 0)$.

\begin{figure}[!htbp]
	\begin{minipage}{0.48\linewidth}
		\begin{center}
			\includegraphics[width=7cm]{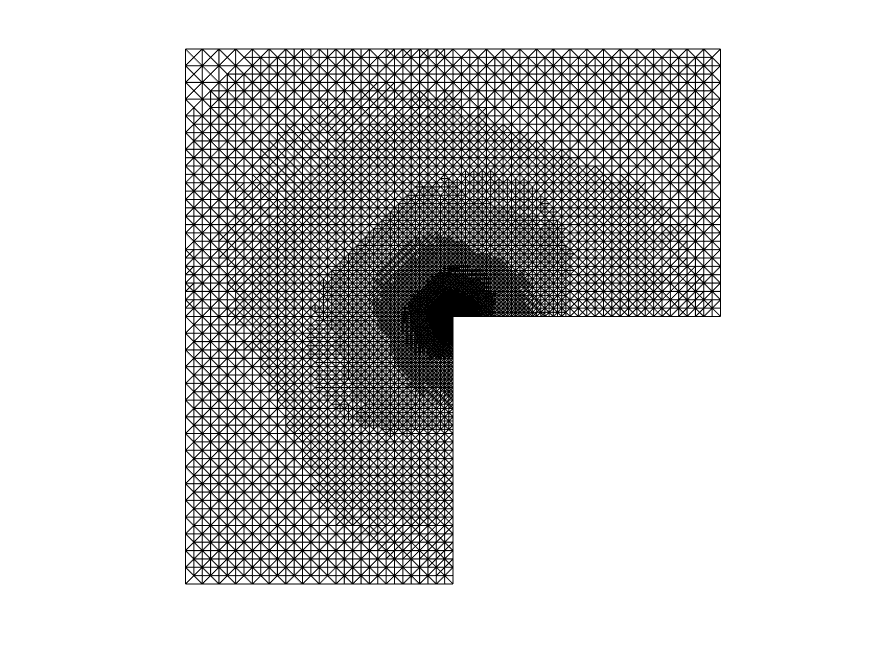}\\
			%	\vspace{-0.3cm}
			(a) Adaptive mesh.
		\end{center}
	\end{minipage}
	\begin{minipage}{0.48\linewidth}
		\begin{center}
			\includegraphics[width=7cm]{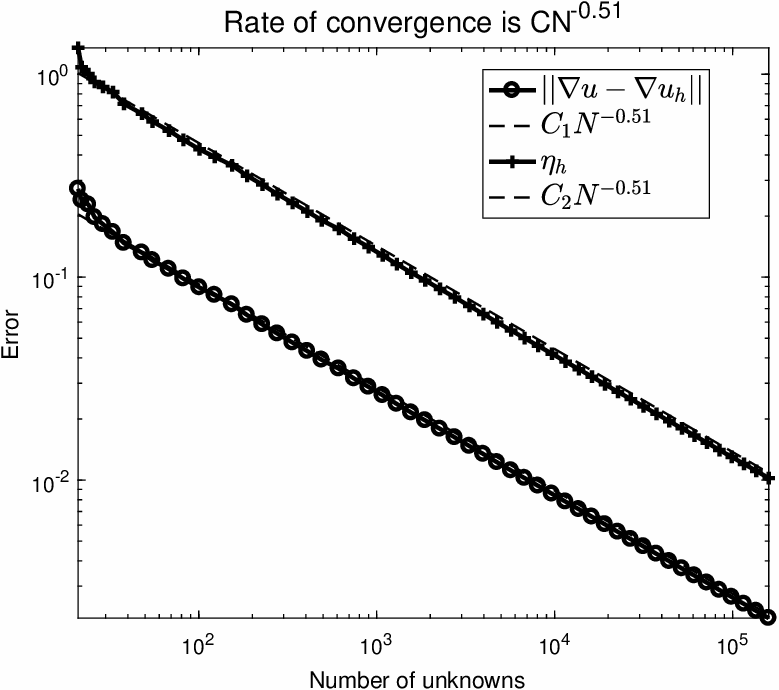}\\
			%	\vspace{0.2cm}			
			(b) Error and the residual estimator.
		\end{center}
	\end{minipage}
	\caption{Local refined mesh and the history of the error and estimator.}\label{residual_result}
\end{figure}

\begin{center}
\begin{table}[htbp]
\begin{tabular}{lllllr} \hline
      $k$    &    $N$       & $||\nabla u-\nabla u_h||$ & $\eta$                  \\ \hline
       1     &     21       & $2.7204933666149056e-01$  & $1.3498972825307682e+00$ \\
       $\vdots$   &     $\vdots$      & $\vdots$                       & $\vdots$                      \\
       34    &     6781     & $1.0225214483815645e-02$  & $5.0066913395745878e-02$ \\
\pmb{35}    & \pmb{8073} & $\pmb{9.3703641684631644e-03}$  & $\pmb{4.5878237026149953e-02}$ \\
       36    &     9661     & $8.5594967357026843e-03$  & $4.1932448254810677e-02$ \\
       $\vdots$   &     $\vdots$      & $\vdots$                       & $\vdots$                      \\
       52    &     135971   & $2.2637731869184861e-03$  & $1.1088270194150417e-02$ \\
\pmb{53} &\pmb{159679} & $\pmb{2.0875642635971403e-03}$ & $\pmb{1.0227397197614970e-02}$ \\\hline
       % G   & $/$        & $/$       & $/$       & $/,~/$    &  / \\ \hline
\end{tabular}

\caption{The data of mesh nodes, error and error estimator.}\label{Lshpaedata}
\end{table}
\end{center}

We apply the residual error estimator \eqref{res} and the newest-vertex bisection refinement method \cite{Mitchell1989ACO} in the adaptive algorithm, and the error tolerance is $0.01$. 
The results are reported in Figure \ref{residual_result}. 
Firstly, it shows that the error estimator successfully guides the mesh refinement near the reentrant corner $(0, 0)$, and the adaptive algorithm
achieves the optimal convergence rate for the gradient error.
Secondly, the error estimator is approximately five times the exact one. This means that it can not drive the adaptive algorithm to stop timely.
To see this clearly, we list the data in Table \ref{Lshpaedata}.
The residual estimator requires a mesh with $159679$ nodes. 
But we note that the exact error is below the tolerance when the mesh with $8073$ nodes. 
This means that the adaptive iterations after the $35$th iteration are not necessary, which wastes a large computational cost. 
\end{example} 

Example \ref{exam3_1} shows that the adaptive method with the residual type error estimator is inefficient due to the unknown constant $C$ in \eqref{res_estimate}. 
In general, $C$ depends on the mesh quality and the partial differential equation. One way to avoid this difficulty is developing a constant-free estimator, such as the equilibrated residual error estimator \cite{cz2012}.

In practice, a posteriori error estimator based on gradient recovery usually provides efficient and reliable error estimates for adaptive algorithms. 
Note that the gradient of linear finite element solution is piecewise constant, gradient recovery smooths it into the linear finite
element space via weighted averaging, projection, or local least square fitting. Denote the set of nodes by $\mathcal{N}$, for every mesh node $z\in \mathcal{N}$, $\phi_z$ is the Lagrange finite element basis function of $V_h$, $\omega_z=supp(\phi_z)$. Let the recovery space be $W_h=V_h\times V_h$ and the gradient recovery operator as $G: V_h\rightarrow W_h$. The recovered gradient  $G(\nabla u_h)$ \cite{hy} based on weighted averaging is defined as 
\[G(\nabla u_h)=\sum_{z\in\mathcal{N}}G(\nabla u_h)(z)\phi_z,\qquad G(\nabla u_h)(z)=\sum_{j=1}^{J_z}
\alpha_z^j(\nabla u_h)_{K_z^j},\quad \forall\,u_h\in V_h^g,\]
where $\cup_{j=1}^{J_z}\bar{K}_z^j=\bar{\omega}_z$, $\sum_{j=1}^{J_z}\alpha_z^j=1, 0\leq\alpha_z^j
\leq 1, j=1,...,J_z$. 
Ideally, gradient recovery $G(\nabla u_h)$ provides a better approximation of the true gradient $\nabla u$ than $\nabla u_h$ does, which means that there exists $\beta<1$ such that
\begin{equation}\label{SAbeta}
    \|\nabla u-G(\nabla u_h)\|\leq \beta\|\nabla u-\nabla u_h\|.
\end{equation}
Using the triangle inequality, we immediately obtain 
\begin{equation*}
	\frac{1}{1+\beta}\|G(\nabla u_h)- \nabla u_h\|\leq\|\nabla u - \nabla u_h\|
	\leq\frac{1}{1-\beta}\|G(\nabla u_h)- \nabla u_h\|.
\end{equation*}
Then the difference between $\nabla u_h$ and recovered gradient $G(\nabla u_h)$ provides a reasonable local and global error estimator $\|G(\nabla u_h)- \nabla u_h\|_{0, K}$ and $\|G(\nabla u_h)- \nabla u_h\|$, respectively. 
If the gradient recovery method is superconvergent, we have $\beta\rightarrow 0$ and 
\begin{equation*}
	\frac{\|G(\nabla u_h)-\nabla u_h\|}{\|\nabla u-\nabla u_h\|}\rightarrow 1,
\end{equation*}
which indicates that the gradient recovery-based error
estimator is asymptotic exact.

Unfortunately, we can not prove \eqref{SAbeta} with $\beta>1$. Similar to \cite{mn}, we present a concrete example to show that the gradient recovery method is invalid and the corresponding error estimator provides no information on the exact error.

\begin{example}\label{exam3_2} [Gradient recovery method can not provide a better approximation.]
Let $\Omega=(0,1)\times(0,1)$, we consider the Poisson equation with homogeneous Dirichlet boundary condition
 and   
\[f(x,y)=\begin{cases}
	~1 &(x,y)\in (\frac{j}{4},\frac{j+1}{4})\times(\frac{k}{4},\frac{k+1}{4}),~~\,j,k=0,1,2,3~
	\text{and} ~ j+k ~\text{is odd}, \\
	-1 & \text{else}.
\end{cases}\] 

\begin{figure}[ht]
	\centering
	\begin{minipage}{0.48\linewidth}
		\begin{center}
			\includegraphics[height=5.5cm,width=6cm]{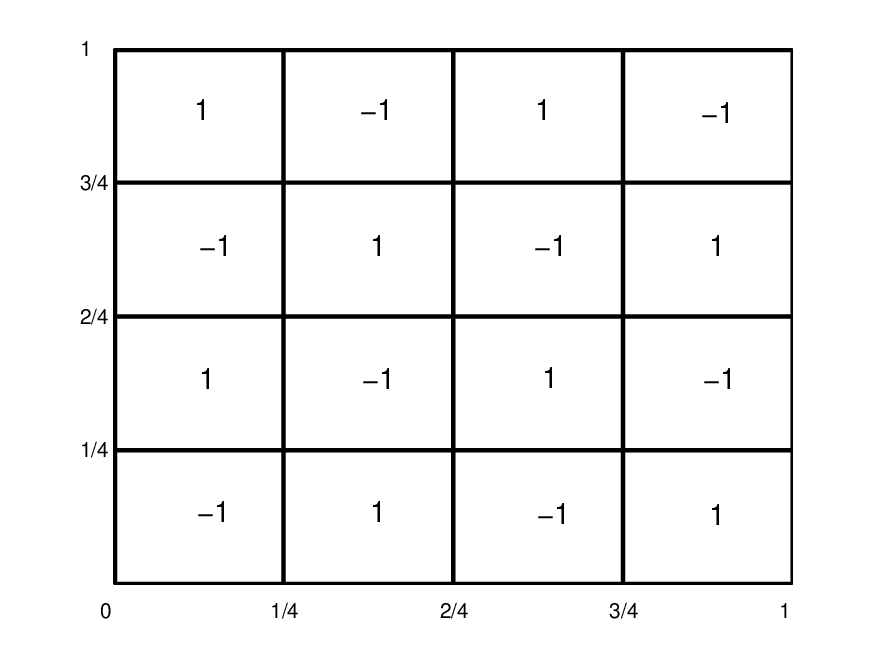} 
		\end{center}
	\end{minipage}
	\begin{minipage}{0.48\linewidth}
		\begin{center}
			\includegraphics[width=7cm]{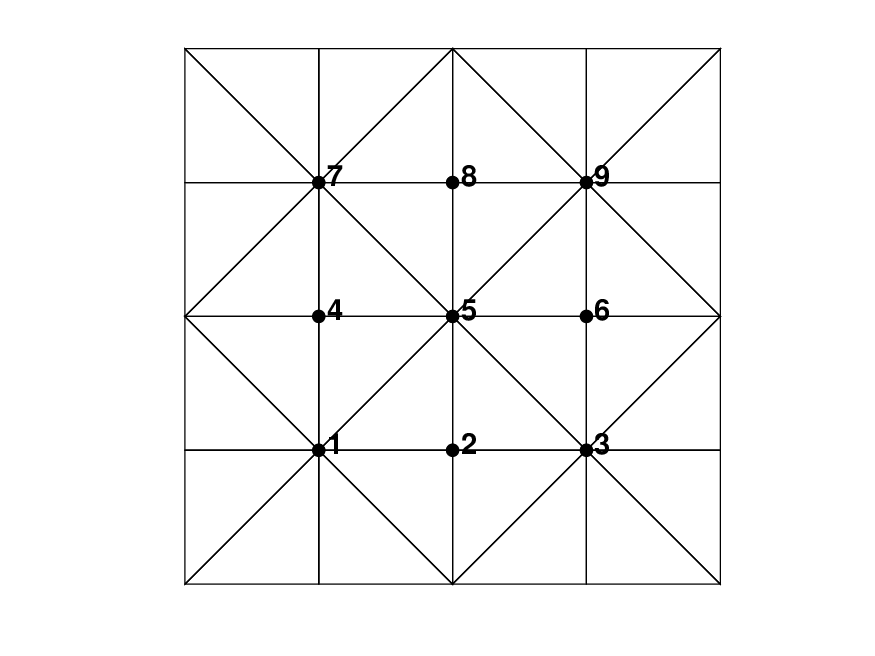} 
		\end{center}
	\end{minipage}
	\caption{Distribution of $f$ and mesh}\label{exam31}
\end{figure}

Figure \ref{exam31} shows the distribution of $f$ and a mesh partition of $\Omega$. If we solve the Poisson equation by the linear finite element method with this mesh, we can easily verify that $f$ is orthogonal to the finite element space, and then the resulting algebraic equation is $AU=0$. Hence, the finite element solution is $u_h=0$, and 
\[\nabla u_h=G(\nabla u_h)=\begin{pmatrix}0\\0\end{pmatrix}.\]
If we take $\|G(\nabla u_h)-\nabla u_h\|$ as the error estimator, it will tell us that the linear finite element solution $u_h=0$ is the exact solution. 
Clearly, $u=0$ is definitely not the solution of the Poisson equation with a piecewise constant $f$. 

In practice adaptive computing, if we take the initial mesh shown in Figure \ref{yixiang} (a),  and set the error tolerance to be $10^{-3}$, the numerical results are displayed in Figure \ref{yixiang}. 
The gradient recovery-based error estimator $\eta=\|G(\nabla u_h)-\nabla u_h\|$ drives the adaptive algorithm to stop at a mesh as shown in Figure \ref{yixiang} (b), and the corresponding numerical approximation is $u_h=0$.

\begin{figure}[!htbp]
	\begin{minipage}{0.48\linewidth}
		\begin{center}
			\includegraphics[width=7cm]{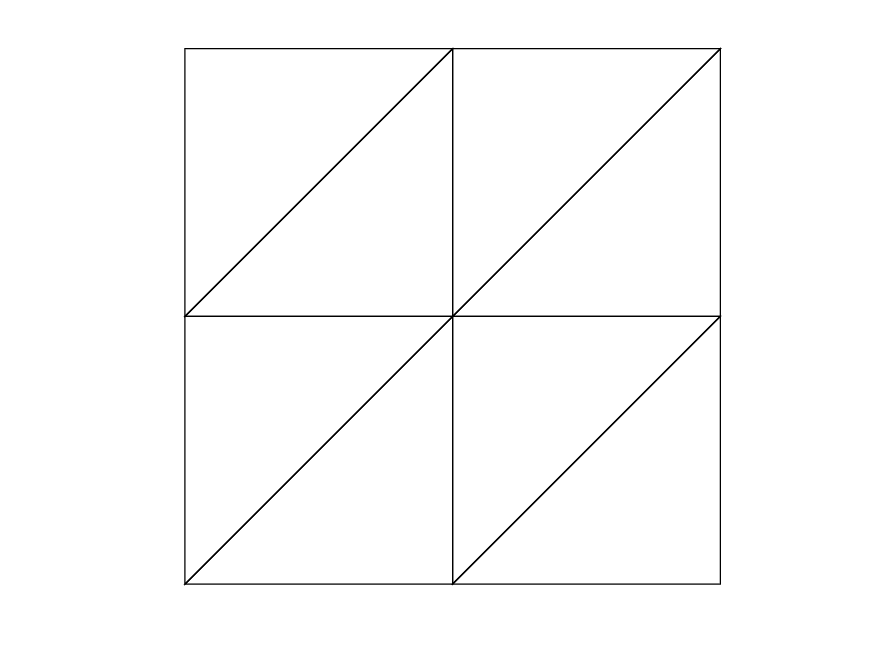}\\
			%\vspace{-0.3cm}
			(a) Initial mesh.
		\end{center}
	\end{minipage}
	\begin{minipage}{0.48\linewidth}
		\begin{center}
			\includegraphics[width=7cm]{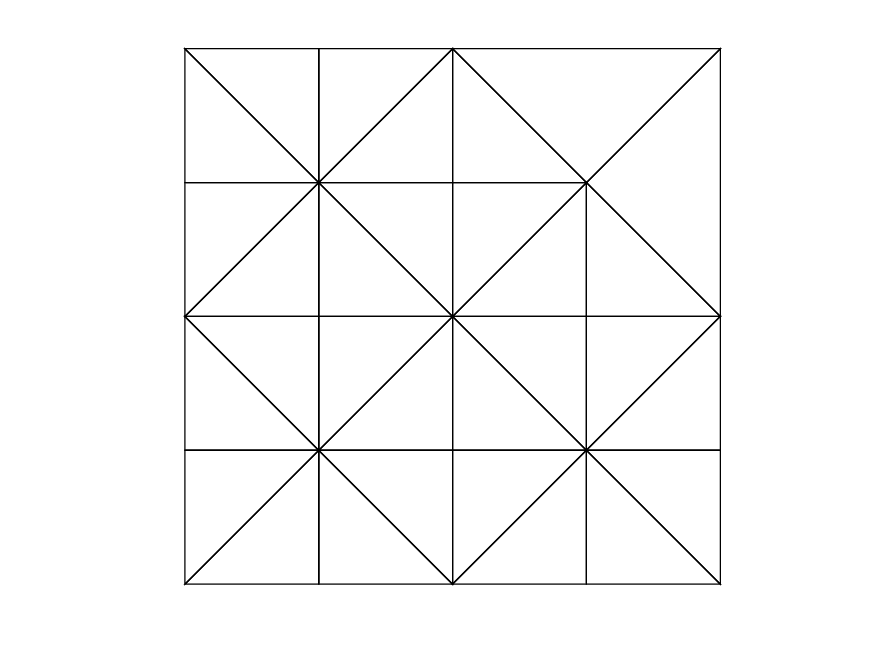}\\
			%\vspace{-0.3cm}
			(b) Final adaptive mesh.
		\end{center}
	\end{minipage}
	
	%	\vspace{0.2cm}
	\begin{minipage}{0.48\linewidth}
		\begin{center}			
			\includegraphics[width=6cm]{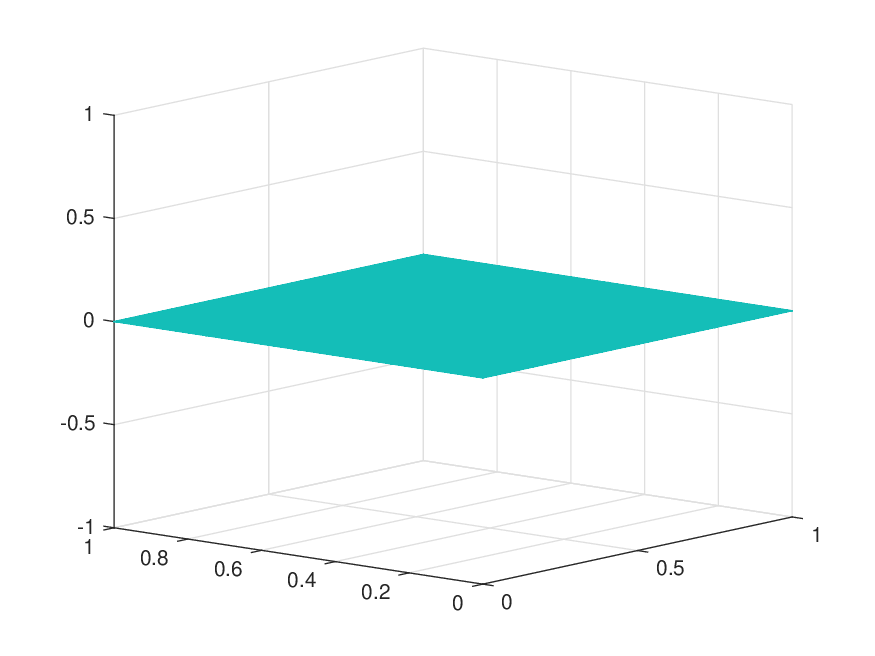}\\
			%		\vspace{-0.2cm}
			(c) Numerical solution on the final mesh.
		\end{center}
	\end{minipage}
	\begin{minipage}{0.48\linewidth}
		\begin{center}	
			\includegraphics[width=6cm]{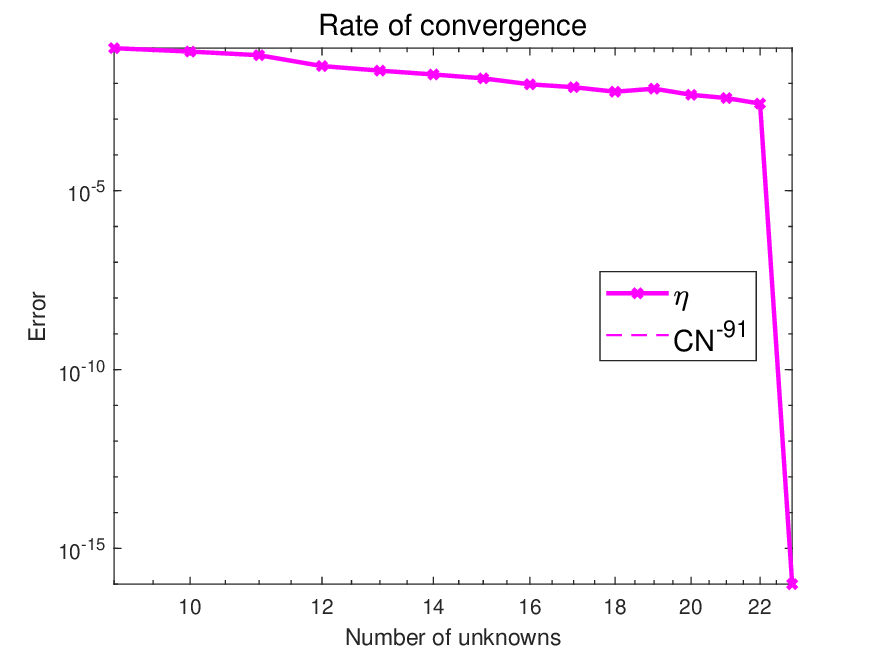}\\
			%\vspace{-0.2cm}
			(d) History of the estimator $\|G(\nabla u_h)-\nabla u_h\|$.
		\end{center}
	\end{minipage}
	\caption{Numerical results of adaptive algorithm for Example \ref{exam3_2}.}\label{yixiang}
\end{figure}

\end{example}

In \cite{Ovall2006}, Ovall demonstrated that a gradient recovery-based estimator can perform arbitrarily poorly for diffusion problems with discontinuous coefficients. 
In the following example, we numerically investigate the performance of the recovery-based estimator to elliptic interface problems.

\begin{example}\label{exam3_3} [Gradient recovery-based error estimator leads to over-refinement.]
Consider the elliptic interface problem 
\[
\left\{
\begin{aligned}
-\nabla \cdot (A\nabla u) = 0 \qquad & in\quad \Omega, \\
u = g \qquad &on \quad \partial \Omega,
\end{aligned}
\right.\]
where $\Omega=[-1,1]\times[-1,1]=\Omega_1\cup\Omega_2\cup\Omega_3\cup\Omega_4$ is as shown in Figure \ref{00} (a). 
The coefficient $A$ and the exact solution $u$ are chosen as
\[A=\begin{cases}
	~1,&\quad (x,y)\in \Omega_1, \\
	10,&\quad (x,y)\in \Omega_2, \\
	100,&\quad (x,y)\in \Omega_3, \\
	1000,&\quad  (x,y)\in \Omega_4,
\end{cases}\qquad
u(x,y)=\begin{cases}
	~x+y+\frac{9}{10},&\quad (x,y)\in \Omega_1, \\
	\frac{1}{10}x+\frac{1}{10}y,&\quad (x,y)\in \Omega_2, \\
	\frac{1}{100}x+\frac{1}{100}y,&\quad (x,y)\in \Omega_3, \\
	\frac{1}{1000}x+\frac{1}{1000}y+\frac{9}{1000},&\quad  (x,y)\in \Omega_4.
\end{cases}\]

We apply the adaptive finite element method with a gradient recovery-based error estimator to solve this interface problem. Set the error tolerance to be $10^{-2}$ and the initial mesh is shown in Figure \ref{00} (b). It is easy to check that the linear finite element solution $u_h=u$, and no error appears. But, as shown in Figure \ref{00} (c)-(d), the error estimator overestimates the local errors and leads to over-refinement.
To improve the reliability and efficiency of gradient recovery error estimators for problems with jumping diffusion coefficients, it is necessary to
avoid averaging across the interfaces between regions with different coefficients. One possible way is to apply the gradient recover method
on each sub-domain $\Omega_j$ separately, then the recovered gradient is not globally continuous.

\begin{figure}[h]
	\begin{minipage}{0.48\linewidth}
		\begin{center}
			\includegraphics[height=5.5cm,width=6cm]{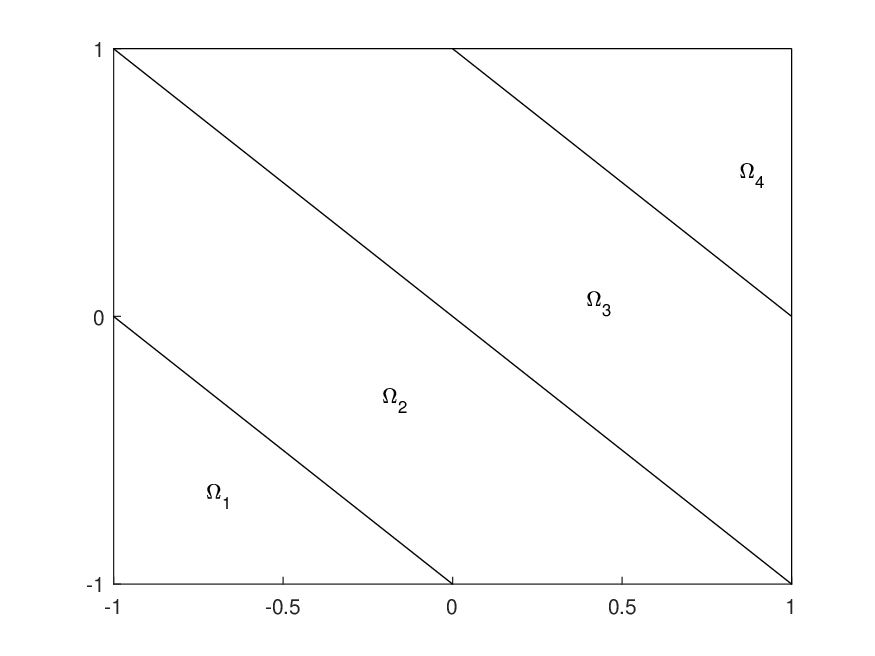}\\
			%	\vspace{-0.3cm}
			(a) Domain $\Omega$.
		\end{center}
	\end{minipage}
	\begin{minipage}{0.48\linewidth}
		\begin{center}
			\includegraphics[width=7cm]{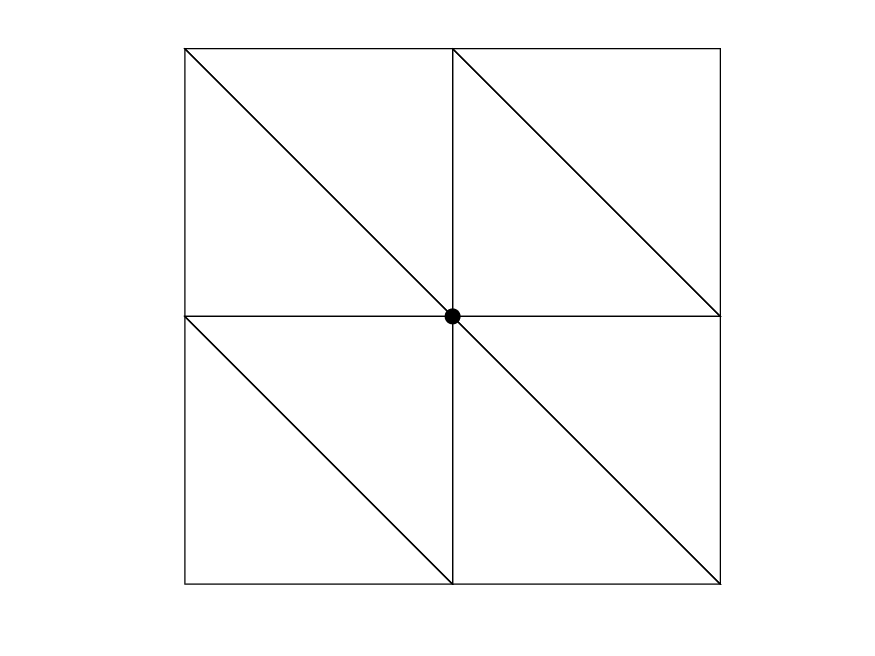}\\
			%	\vspace{-0.3cm}
			(b) Initial mesh.
		\end{center}
	\end{minipage}
	
	\begin{minipage}{0.48\linewidth}
		\begin{center}
			\includegraphics[width=7cm]{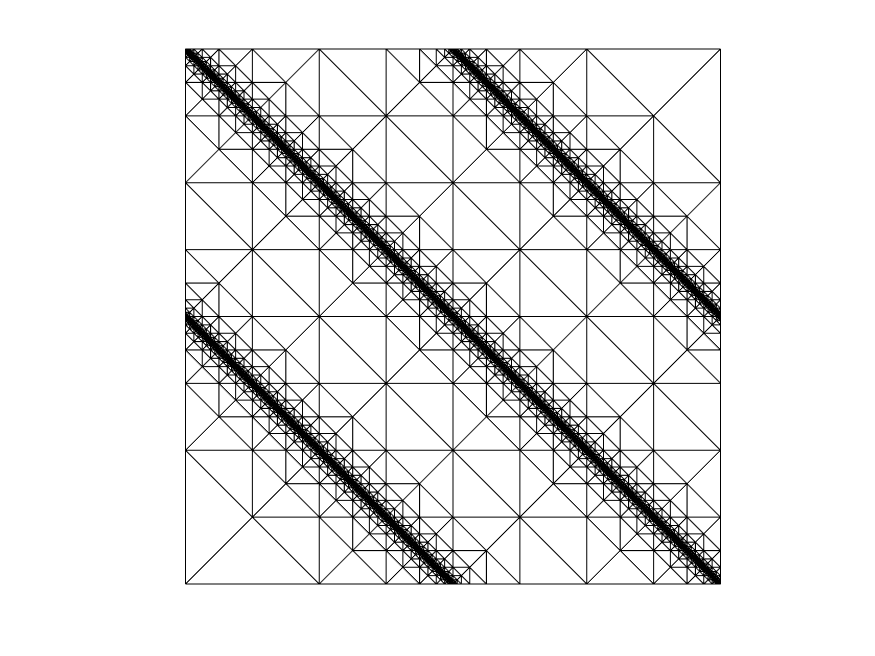}\\
			%	\vspace{-0.3cm}
			(c) Adaptive mesh.
		\end{center}
	\end{minipage}
	\begin{minipage}{0.48\linewidth}
		\begin{center}
			\includegraphics[width=7cm]{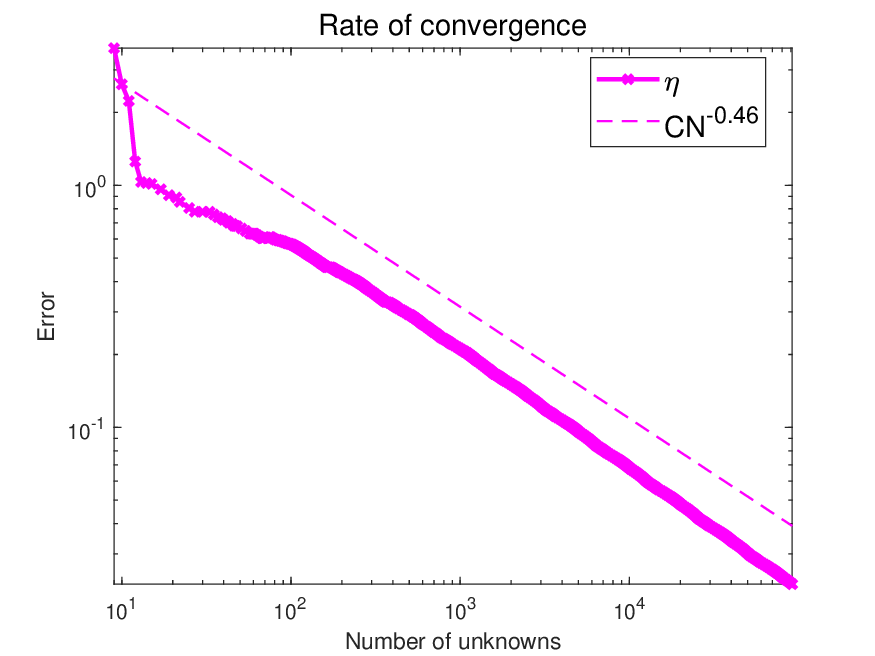}\\
			%	\vspace{-0.3cm}
			(d) Error estimator.
		\end{center}
	\end{minipage}
	\caption{Numerical results of Example \ref{exam3_3}.}\label{00}
\end{figure}

\end{example}

In the three examples above, we have shown that the classical error estimators can be arbitrarily poor and that the local error estimators can lead to over-refinement or produce a wrong finite element approximation. In the following sections, we provide an improved gradient recovery-based error estimator, which can avoid the aforementioned difficulties encountered by the classical error estimators. 

\subsection{Improved error estimator based on gradient recovery}
For ease of presentation, we restrict ourselves to the homogeneous Dirichlet boundary value problem for the Poisson equation
\begin{equation}\label{mp1} 
\left\{
\begin{aligned}
-\Delta u = f \qquad & in\quad \Omega, \\
u = 0 \qquad &on \quad \partial \Omega.
\end{aligned}
\right.   
\end{equation}
Its finite element scheme reads: find $u_h\in V_h^0$ such that
\begin{equation}\label{mpfem}
    (\nabla u_h, \nabla v_h)=(f, v_h),\qquad \forall v_h\in V_h^0.
\end{equation}
Define the local gradient recovery-based a posteriori error estimator $\eta_K$ on $ K \in \mathcal{T}_h$ as
\begin{equation}\label{estomator}
	\eta_K^2=\|G(\nabla u_h)-\nabla u_h\|^2_{0,K}+h_K^2\|{f}+\nabla\cdot G(\nabla u_h)\|_{0,K}^2.
\end{equation}
Summing the above equation over all element $K$ in $\mathcal{T}_h$, the global error estimator $\eta$ is
\[\eta^2=\sum_{K\in \mathcal{T}_h }\eta_K^2=\|G(\nabla u_h)-\nabla u_h\|^2+\|h({f}+\nabla\cdot G(\nabla u_h))\|^2.\]
We now establish the reliability and efficiency of the improved error estimator, i.e., it is a global upper and local lower bound of the exact error.

Denote by $I_h: L^2(\Omega)\rightarrow V_h^0$ the quasi-interpolation operator of
Cl{\'e}ment \cite{verfurth1999}. We have the following interpolation error estimations for $I_h$.

\begin{lemma}[\cite{verfurth1999}]\label{lem1}
	For any $K\in \mathcal{T}_h$, $v\in H_{0,D}^1(\Omega)$, $\omega_K\equiv\cup_{\bar{K'}\cap\bar{K}\neq\emptyset}{K'}$, there exists a positive constant $C$ such that
	\begin{align}\label{17a}
		\|v-I_hv\|_{0,K}\leq Ch_K\|\nabla v\| _{\omega_K},\\
		\label{17b}
		\|\nabla(v-I_hv) \|_{0,K}\leq C\|\nabla v\|_{\omega _K}.
	\end{align}
\end{lemma}

\begin{theorem}\label{cthe1}
	Let $u$ and $u_h$ be the solution of \eqref{mp1} and \eqref{mpfem}, respectively. Then there exists a  positive constant $C_r$, such that the estimator $\eta$ defined in \eqref{estomator} satisfies the following global reliability
	\begin{equation}\label{kekao}
		\|\nabla u-\nabla u_h\| ^2\leq C_r\sum_{K\in\mathcal{T}_h}\left(\|G(\nabla u_h) -\nabla u_h\|_{0,K}^2
		+ h_K^2\|f+\nabla\cdot  G(\nabla u_h) \|_{0,K}^2\right).
	\end{equation}
\end{theorem}
\begin{proof}
	Let $e=u-u_h$, utilizing the Galerkin orthogonality of the finite element approximation, 
	integration by parts and \eqref{Model problem},  we obtain
	\begin{align*}
		\|\nabla e\| ^2 &=(\nabla e, \nabla e)=(\nabla e, \nabla (e-I_he))=( \nabla(u-u_h),\nabla(e-I_he)) \\
		&=( \nabla u-G(\nabla u_h),\nabla(e-I_he))+(  G(\nabla u_h)-\nabla u_h,\nabla(e-I_he))\\
		&=\sum_{K\in\mathcal{T}_h}({f}+\nabla\cdot G(\nabla u_h),e-I_he)_K+\sum_{K\in\mathcal{T}_h}(  G(\nabla u_h)- \nabla u_h,\nabla(e-I_he))_K\\
		&=:M+N,
	\end{align*}
	where 
	\begin{align*}
		M=&\sum_{K\in\mathcal{T}_h}({f}+\nabla\cdot G(\nabla u_h),e-I_he)_K,\\
		N=&\sum_{K\in\mathcal{T}_h}(  G(\nabla u_h)- \nabla u_h,\nabla(e-I_he))_K.
	\end{align*}
Using the Cauchy-Schwartz inequality and \eqref{17a}, we have
	\begin{equation*}
		\begin{split}
			M&\leq \sum_{K\in\mathcal{T}_h}\|{f}+\nabla\cdot G(\nabla u_h)\|_{0,K}\|e- I_he\|_{0,K}\\
			&\leq  C\left(\sum_{K\in\mathcal{T}_h} h^2_K\|{f}+\nabla\cdot G(\nabla u_h)\|^2_{0,K}\right)^{1/2}
			\|\nabla e\|.
		\end{split}
	\end{equation*}
Using the Cauchy-Schwartz inequality, \eqref{17b} and the Young inequality, we have
	\begin{align*}
		N&\leq\sum_{K\in\mathcal{T}_h}\| G(\nabla u_h)-\nabla u_h\|_{0,K}\|\nabla(e-I_he)\|_{0,K}\\
		&\leq C\left(\sum_{K\in\mathcal{T}_h}\|G(\nabla u_h)-\nabla u_h\|_{0,K}\right)^{1/2}\|\nabla e\|. 
	\end{align*}
Then, we obtain 
	\begin{equation}
		\|\nabla e\|\leq C_r\left(\left(\sum_{K\in\mathcal{T}_h}\|G(\nabla u_h)-\nabla u_h\|_{0,K}\right)^{1/2}+\left(\sum_{K\in\mathcal{T}_h} h^2_K\|{f}+\nabla\cdot G(\nabla u_h)\|^2_{0,K}\right)^{1/2}\right),
	\end{equation}
	which completes the proof.
\end{proof}

Note that all gradient recovery techniques can be viewed as performing some sort of averaging of the piecewise constant gradient $\nabla u_h$. Using the estimation of $\|G(\nabla u_h)-\nabla u_h\|_{0, K}$ obtained by Du and Yan \cite{dy}, we can get a local lower bound of the numerical error in terms of local indicators and oscillation.

\begin{theorem}\label{youxiao}
	Let $u$ and $u_h$ be the solution of \eqref{mp1} and \eqref{mpfem}, $f\in L^2(\Omega)$, respectively. Then there exists a positive constant $C$, which is independent of mesh size $h$, such that the estimator defined in \eqref{estomator} satisfies
	\begin{equation}\label{loe}
		\eta_K^2\leq C\left(\|\nabla   u-\nabla u_h\| _{\omega_K}^2 + \sum_{K\in \omega_K'}{h}_{K'}^2\|f-\bar{f}\|_{0,K'}^2\right),
	\end{equation} 
where $\bar{f}$ denotes a piecewise constant approximation of $f$ over $\mathcal{T}_h$ with
element value $f_K$ equal to the mean value of $f$ over $K\in\mathcal{T}_h$.
\end{theorem}

\begin{proof}
For the edge jump of the classical residual estimator,	from \cite{mn1}, we obtain
	\begin{equation}\label{step2}
		\sum_{e\in \partial K} h_e\|J_e( \nabla u_h)\|_{0,e}^2 \leq C\left(\|\nabla   u-\nabla u_h\|_{\omega_K}^2 + \sum_{K\in\omega_K}h_K^2\|f - \bar{f}\|_{0,K}^2\right).
	\end{equation}
	From \cite{dy}, for the first term of the estimator defined in \eqref{estomator},  we have
	\begin{equation}\label{step3}
		\|G(\nabla u_h) -\nabla u_h\|_{0,K}^2 \leq C\sum_{e\in\partial K}h_e\|J_e( \nabla u_h)\|_{0,e}^2.
	\end{equation}
	Following from  \eqref{step2} and \eqref{step3}, we obtain 
	\begin{equation}\label{step_1}
		\|G(\nabla u_h) -\nabla u_h\|_{0,K}^2 \leq C\|\nabla   u-\nabla u_h\|_{\omega_K}^2 + C\sum_{K\in\omega_K}h_K^2\|f - \bar{f}\|_{0,K}^2.
	\end{equation}	
For the second term of the estimator, using the inverse inequality, we have 
	\begin{align}\label{step_2}
		h_K^2\|f+\nabla\cdot  G(\nabla u_h )\|_{0,K}^2
		=&h_K^2\|f+\Delta u_h
		+\nabla\cdot  G(\nabla u_h) -\nabla\cdot \nabla u_h\|_{0,K}^2\nonumber\\
		\leq& h_K^2\|f+\Delta u_h\|_{0,K}^2+ h_K^2\|\nabla\cdot( G(\nabla u_h) -\nabla u_h)\|_{0,K}^2\nonumber\\
		\leq& h_K^2\|f+\Delta u_h\|_{0,K}^2+ C\|G(\nabla u_h) -\nabla u_h\|_{0,K}^2.
	\end{align} 
	From \cite{mn1}, we know that
	\begin{equation}\label{step1}
		h_K^2\|f + \Delta u_h\|_{0,K}^2 \leq C\left(\|\nabla   u-\nabla u_h\|_{\omega_K}^2
		+\sum_{K\in\omega_K}h_K^2\|f - \bar{f}\|_{0,K}^2\right).
	\end{equation}
Substituting \eqref{step1} and \eqref{step_1} into \eqref{step_2}, we obtain
\begin{equation}\label{step_3}
		h_K^2\|f+\nabla\cdot  G(\nabla u_h) \|_{0,K}^2	\leq C\left(\|\nabla   u-\nabla u_h\|_{\omega_K}^2
		+\sum_{K\in\omega_K}h_K^2\|f - \bar{f}\|_{0,K}^2\right).
\end{equation}
Then \eqref{loe} follows from \eqref{step_1} and \eqref{step_3}.
\end{proof}

\begin{rem}
Our gradient recovery based on a posteriori error estimator and adaptive algorithm can be generalized to the model problem \eqref{Model problem} with some modifications. For $A\neq I$, the local error estimator is defined as:
\begin{equation}\label{estomator1}
	\eta_K^2=\|{A }^{1/2}(G(\nabla u_h)-\nabla u_h)\|^2_{0,K}+h_K^2\|{A }^{-1/2}({f}+\nabla\cdot({A }G(\nabla u_h)))\|_{0,K}^2,
\end{equation}
and the global error estimator
$$\eta^2=\sum_{K\in \mathcal{T}_h }\eta_K^2=\|{A }^{1/2}(G(\nabla u_h)-\nabla u_h)\|^2_{0,\Omega}+\|{A}^{-1/2}h({f}+\nabla\cdot({A }G(\nabla u_h)))\|_{0,\Omega}^2.$$
\end{rem}

\section{Adaptive finite element method}\label{sec4}
In this section, based on the improved gradient recovery type error estimator \eqref{estomator} and the newest-vertex bisection refinement method, we introduce an adaptive finite element algorithm and analyze its convergence.

In the adaptive algorithm, we adopt the D{\"o}rfler’s bulk strategy \cite{marking1996} to control both error and oscillation. We use two mark strategies: Marking Strategy E deals with the
error estimator and Marking Strategy R does so with the oscillation. 

\begin{itemize}
  \item \textbf{Marking strategy E}: Given a parameter $0<\theta_E<1$, construct  a minimal subset $\mathcal{M}_k$ of $\mathcal{T}_k$ such that
      \begin{equation}\label{E}
      \sum_{K\in\mathcal{M}_k}\eta_K^2(k)\geq\theta_E^2\eta^2(k),
      \end{equation}
      where $\eta^2(k)=\sum_{K\in \mathcal{T}_k}\eta_K^2(k)$,  $k$ is used for the level $k$  in the adaptive loops.
  \item \textbf{Marking strategy R}: Given a parameter $0<\theta_0<1$ and the Marking  Strategy E, enlarge $\mathcal{M}_k$ to a minimal set (denoted again by $\mathcal{M}_k$) such that
      \begin{equation}\label{R}
      \sum_{K\in\mathcal{M}_k}h_K^2(k)\|f-\bar{f}\|_{0,K}^2
      \geq\theta_0^2\|h(k)(f-\bar{f})\|_{0,\Omega}^2.
      \end{equation}
\end{itemize}

\begin{algorithm}
\caption{Adaptive Finite Element Algorithm}
\label{afemalg}
\begin{algorithmic}
\STATE Given two parameters $\theta_E, \theta_0\in (0,1)$, initial mesh $\mathcal{T}_0$. Set $k := 0$.

\STATE \textbf{Step 1} (SOLVE) Solve the discrete equation \eqref{mpfem} for numerical solution $u_k$ on $\mathcal{T}_k$.

\STATE \textbf{Step 2} (ESTIMATE) 
\begin{itemize}
\item[$\bullet$] Computing the local error estimator
                \begin{align*}
                \eta_{K}(k)^2 = \|G(\nabla u_k)-\nabla u_k\|^2_{0,K}+\|h(k)(f+\nabla\cdot G(\nabla u_k))\|_{0,K}^2.
                \end{align*}
\item[$\bullet$]Computing the total error estimator by summing over all $K \in \mathcal{T}_k$
                \begin{align*}
                \eta(k) = \left(\sum_{K \in \mathcal{T}_k}\eta_{K}^2(k)\right)^{\frac{1}{2}}.
                \end{align*}\;
\end{itemize}

\STATE \textbf{Step 3} (MARK) Mark a subset $\mathcal{M}_k\subset \mathcal{T}_k$ satisfying \eqref{E} and \eqref{R}.

\STATE \textbf{Step 4} (REFINE) Refine each triangle $\mathcal{M}_k$ by the newest vertex bisection to
get $\mathcal{T}_{k+1}$.

\STATE \textbf{Step 5} Set $k := k + 1$ and go to Step $1$.
\end{algorithmic}
\end{algorithm}

The new proposed recovery type estimator and the newest vertex bisection refinement with the interior property are employed in the adaptive algorithm. The convergence analysis of the adaptive algorithm presented here is similar to that of \cite{cz}. To establish the convergence of the adaptive method, we start with the following Lemmas.

\begin{lemma}\cite{mn}\label{hf}
Let $0<\gamma_0<1$ be the reduction factor of element size associated with one refinement and $\theta_0$ is given in the Marking Strategy R. Let $\mathcal{M}_k$ be a subset of $\mathcal{T}_k$ satisfying Marking Strategy R, if $\mathcal{T}_{k+1}$ is generated by the refinement from $\mathcal{T}_k$, then 
\begin{equation}\label{bj}
 \|h(k+1)(f-\bar{f})\|\leq\zeta\|h(k)(f-\bar{f})\|,
\end{equation}
where $\zeta=\sqrt{1-(1-\gamma_0^2)\theta_0^2}$.
\end{lemma}

\begin{lemma}\cite{mn}\label{ff}
For the  residual type a posteriori error estimator defined in \eqref{res}, if each $e\in\mathcal{E}_k$ contains a vertex of $\mathcal{T}_{k+1}$ as its interior point, there exists a positive constant $C$ such that 
\begin{equation}\label{bianjie111}
\eta_{res,K}^2(k)\leq C\left(\|\nabla (u_{k+1}-u_k)\|_{\omega_e}^2+\|h(k)(f-\bar{f})\|_{0,\omega_e}^2\right),
\end{equation}
where $\eta_{res,K}^2(k)=\|h(k)f\|_{0,K}^2+	\sum\limits_{e\in \partial K} h_e(k)\|J_e(\nabla u_k)\|_{0,e}^2$
\end{lemma}
 
\begin{lemma}\label{houyan}There exists a positive constant $C_l>0$ such that
  \begin{align}\label{biaoji}
  {{\eta}}^2(k)\leq C_l\left(\|\nabla (u_{k+1}-u_k)\|^2+\|h(k)(f-\bar{f})\|^2\right).
  \end{align}
\end{lemma}

\begin{proof}
Using  the estimations \eqref{step_1}, \eqref{step_3}, \eqref{res_global} and \eqref{bianjie111}, we obtain
  \begin{align*}%\label{biaoji}
  {{\eta}}^2(k)&=\|G(\nabla u_k)-\nabla u_k\|^2+\|h(k)(f+\nabla\cdot G(\nabla u_k))\|^2\\
  &\leq  C\left(\|u-u_k\|^2+\|h(k)(f-\bar{f})\|^2\right)\\
  &\leq C\left(\sum\limits_{K\in\mathcal{T}_k}\eta_{res, K}^2(k)+\|h(k)(f-\bar{f})\|^2\right)\\
  &\leq C_l\left(\|\nabla( u_{k+1}-u_k)\|^2+\|h(k)(f-\bar{f})\|^2\right),
  \end{align*}
  which completes the proof.
  \end{proof}

%\begin{coro}\label{coro1}
%For $k=1,2,\ldots$, there exists a positive constant $C_r$ such that
%  \begin{equation}\label{ar}
%  \|\nabla (u-u_k)\|^2\leq C_r\eta^2(k).
%  \end{equation}    
%\end{coro} 
 
\begin{lemma}\label{udiedai}
  Under the conditions of Lemma \ref{houyan} and Theorem \ref{cthe1}, it holds that
\begin{equation}\label{are}
  \|\nabla (u_{k+1}-u_k)\|^2\geq\delta_1\|\nabla (u-u_k)\|^2-
  \|h(k)(f-\bar{f})\|^2.
\end{equation}
 with $\delta_1=\frac{1}{C_lC_r}$.
\end{lemma}

\begin{proof}
From Theorem \ref{cthe1}, for $u_k\in V_h^0(k)$, there exists a positive constant $C_r$ such that
\[
\eta^2(k)\geq \frac{1}{C_r}\|\nabla (u-u_k)\|^2.
\]
According to \eqref{biaoji}, we have 
\[\begin{aligned}
\|\nabla (u_{k+1}-u_k)\|^2\geq &\frac{1}{C_\ell}\eta^2(k)- \|h(k)(f-\bar{f})\|^2\\
\geq &\frac{1}{C_\ell C_r}\|\nabla (u-u_k)\|^2- \|h(k)(f-\bar{f})\|^2.
\end{aligned}\]
\end{proof}
 
\begin{theorem}
Set $\delta=\max\{\sqrt{1-\delta_1},\zeta\}$, $\delta_1<1$ and $\zeta$ are given in Lemma $\ref{udiedai}$ and Lemma $\ref{hf}$. Under the condition of the Lemma \ref{udiedai}, the sequence \{$u_k$\} generated by the adaptive finite element algorithm satisfies 
\begin{equation}\label{ss}
  \|\nabla (u-u_k)\|\leq C_0\delta^k
\end{equation}
 with $C_0=\sqrt{\|\nabla( u-u_0)\|^2+\|h(0)(f-\bar{f})\|^2}$.
\end{theorem}

\begin{proof}
Since $V_h^0(k)\subset V_h^0(k+1)$, from the orthogonality of the finite element approximations and  Lemma \ref{udiedai}, it holds  that
\begin{align*}
  \|\nabla (u-u_k)\|^2&=\|\nabla (u-u_{k+1})\|^2+\|\nabla (u_{k+1}-u_k)\|^2\\
  &\geq \|\nabla (u-u_{k+1})\|^2+\delta_1 \|\nabla (u-u_k)\|^2-\|h(k)(f-\bar{f})\|^2,
\end{align*}
then we have
\begin{equation}\label{shoulian}
  \|\nabla (u-u_{k+1})\|^2\leq(1-\delta_1) \|\nabla (u-u_k)\|^2+\|h(k)(f-\bar{f})\|^2.
\end{equation}
Let $\textbf{y}_k=(\|\nabla (u-u_k)\|,\|h(k)(f-\bar{f})\|)^T$, and
\[B=\begin{pmatrix}
            (1-\delta_1)^{1/2} & 1 \\
            0 & \zeta
          \end{pmatrix},\]
it follows from \eqref{shoulian} and Lemma \ref{hf} that
$$\textbf{y}_k\leq B\textbf{y}_{k-1}\leq B^k\textbf{y}_0.$$
Hence,
$$\|\nabla (u-u_k)\|\leq \|\textbf{y}_k\|_{l^2}\leq \rho(B^k)\|\textbf{y}_0\|_{l^2}.$$
Further, $\rho(B^k)=\max\{(1-\delta_1)^{k/2},\zeta^k\}=\delta^k$, $C_0=\|\textbf{y}_0\|_{l^2}$,
which leads to the proof.
\end{proof}

\section{Numerical experiments}\label{sec5}
\setcounter{equation}{0}
We test the performance of the proposed gradient recovery-based error estimator \eqref{estomator} and adaptive algorithm with six examples. 
We first investigate the performance of the improved error estimator in Example \ref{exam3_1} and Example \ref{exam3_2}, where the performance of classical residual type and gradient recovery type estimators are very poor. 
Then, we numerically test the performance of the improved error estimator on problems with the inner layer or in the three-dimensional concave domain, where the solutions show line singularity. 
At last, we extend our error estimator to the model problem \eqref{Model problem} with continuous variable or piecewise constant coefficients $A$.

\begin{example} \label{test2}
Apply the adaptive finite element algorithm with gradient recovery-based error estimator \eqref{estomator} to Example \ref{exam3_1}. 
Figure \ref{test_2} presents the numerical solution, adaptive mesh, history of the error and estimator. Similar to Figure \ref{residual_result} (a), which is obtained with the classical residual type error estimator, error estimator \eqref{estomator} also guides the mesh locally refined at the corner $(0,0)$, aligning with the singularity of the solution. Figure \ref{test_2} (c) plots the history of gradient error $\|\nabla u-\nabla u_h\|_{0,\Omega}$ and error estimator $\eta$. Note that the error estimator $\eta$ contains two parts, $\|\nabla u_h-G(\nabla u_h)\|$ and  $\| \nabla u-  G(\nabla u_h)\|$, we also plot the two terms separately to examine which part dominates the error estimator. It can be seen clearly that the error estimator is asymptotically exact, which in turn drives the adaptive algorithm to achieve the quasi-optimal convergence rate $O(N^{-1/2})$. Moreover, the recovered gradient $G(\nabla u_h)$ is superconvergent to $\nabla u$, thus $\|\nabla u_h-G(\nabla u_h)\|$ is close to error estimator $\eta$ and accounts for the major part. Moreover, the term $\|h(f+\nabla \cdot G(\nabla u_h))\|$ of error estimator $\eta$ is much smaller and also superconvergent. 

\begin{figure}[!htbp]
	\begin{minipage}{0.48\linewidth}
		\begin{center}
			\includegraphics[width=7cm]{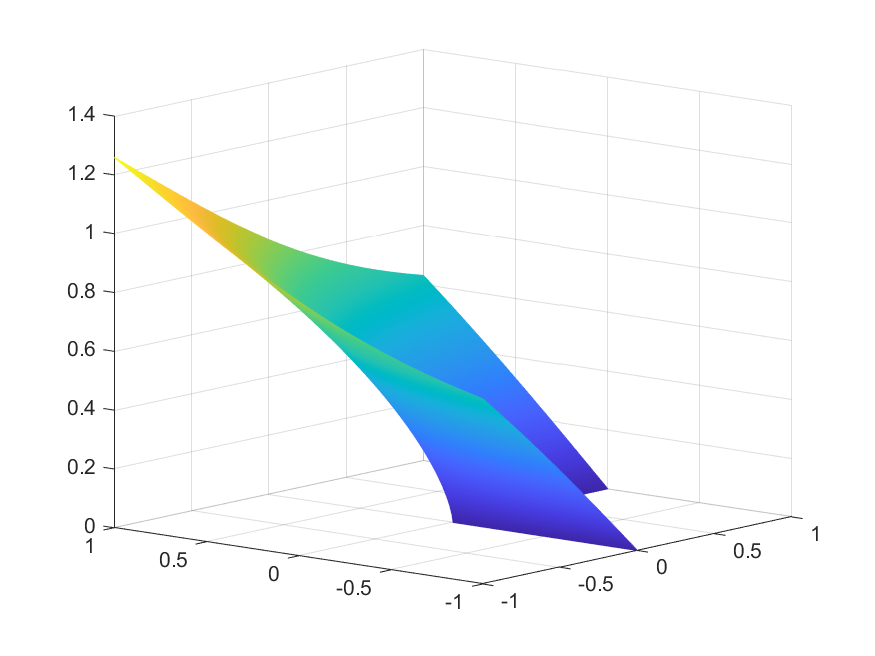}\\
			(a) Numerical solution.
		\end{center}
	\end{minipage}
	\begin{minipage}{0.48\linewidth}
		\begin{center}
			\includegraphics[width=7cm]{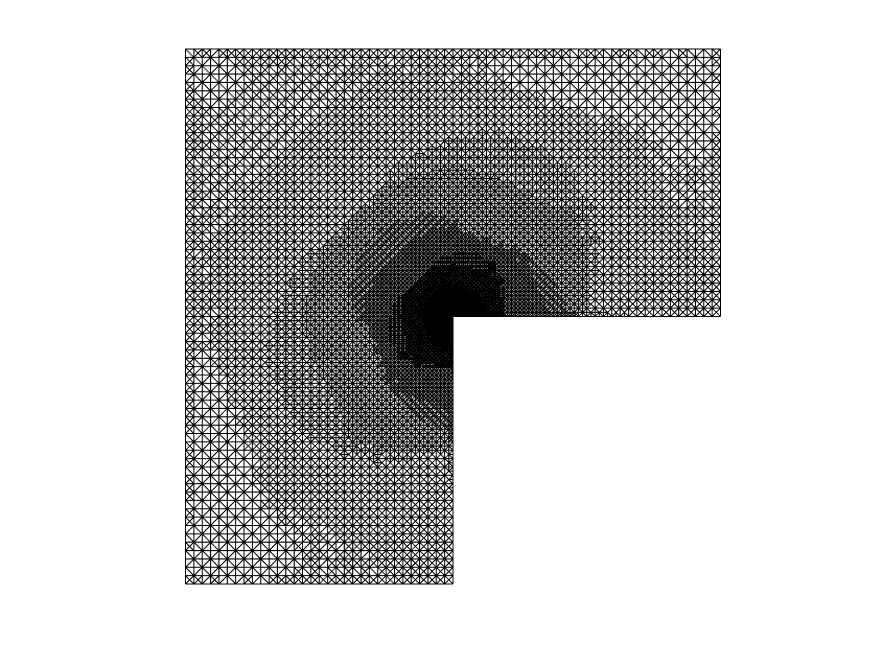}\\
			(b) Adaptive mesh.
		\end{center}
	\end{minipage}
	\begin{minipage}{0.48\linewidth}
		\begin{center}			
			\includegraphics[width=6cm]{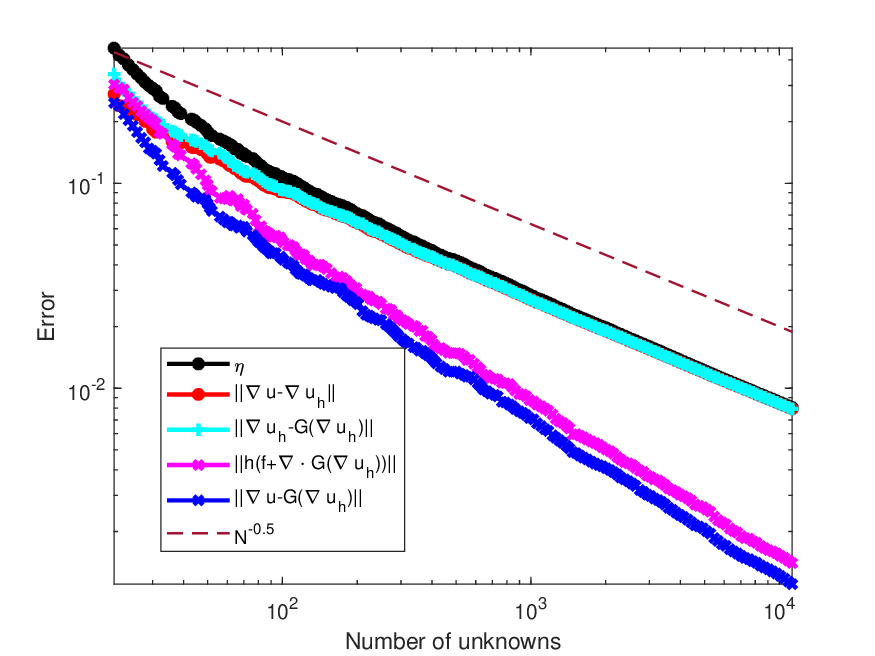}\\
			(c) History of error and error estimator.
		\end{center}
	\end{minipage}
	\caption{Numerical results of Example \ref{test2} with error estimator \eqref{estomator}.}\label{test_2}
\end{figure}

\end{example}

\begin{example} \label{test1}
Example \ref{exam3_2} shows that the classical gradient recovery-based error estimator $\|G(\nabla u_h)-\nabla u_h\|$ does not work for this specific case. 
We now apply our improved error estimator to this example to show the necessity of the term $\|h({f}+\nabla\cdot(A G(\nabla u_h))\|$ in recovery type error estimator.  
Figure \ref{test_1} displays the initial mesh, adaptive mesh, numerical solution, and error estimator, respectively. In comparison to the numerical results of Example \ref{exam3_2}, our improved error estimator can derive the adaptive algorithm to avoid the state that gradient recovery can not provide better gradient approximation. Figure \ref{test_1} (d) shows that $\|\nabla u_h-G(\nabla u_h)\|$ dominates $\eta$,  and $\|h(f+\nabla\cdot G(\nabla u_h))\|$ is superconvergent.
 
\begin{figure}[!htbp]
	\begin{minipage}{0.48\linewidth}
		\begin{center}
			\includegraphics[width=7cm]{figure/initialmeshnocolor.eps}\\			
			(a) Initial mesh.
		\end{center}
	\end{minipage}
	\begin{minipage}{0.48\linewidth}
		\begin{center}
			\includegraphics[width=7cm]{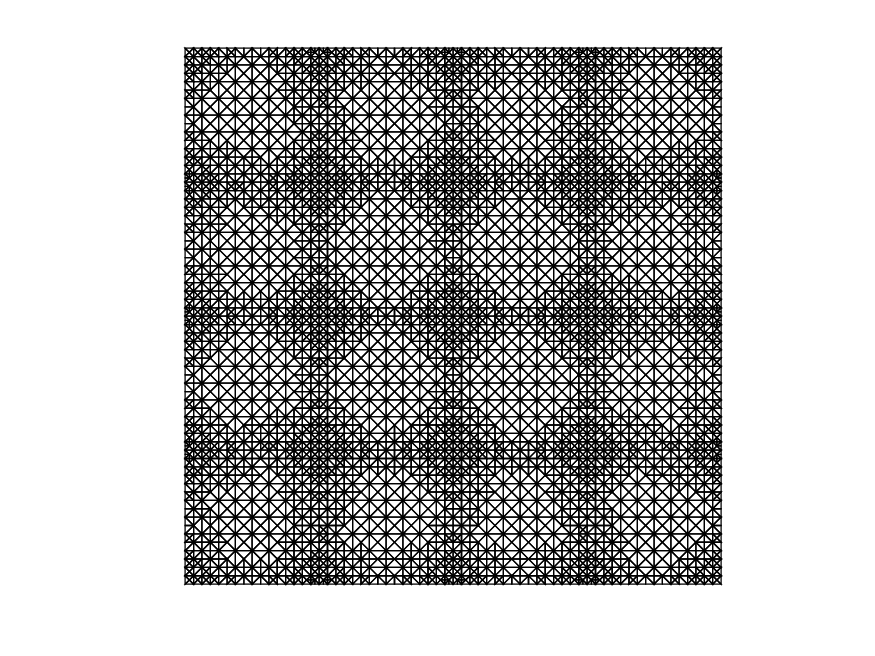}\\			
			(b) Adaptive mesh.
		\end{center}
	\end{minipage} 
	\begin{minipage}{0.48\linewidth}
		\begin{center}			
			\includegraphics[width=6cm]{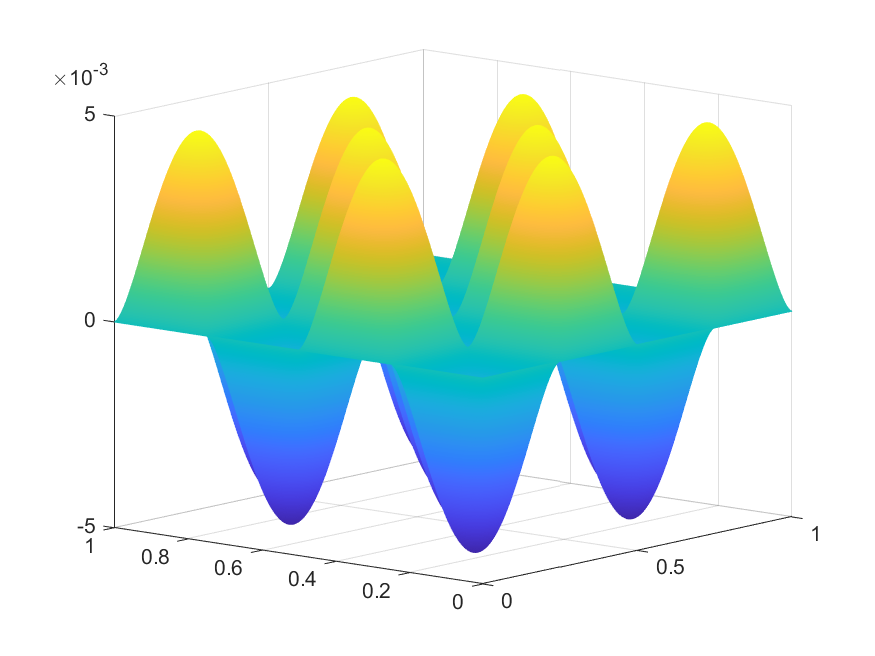}\\
			(c) Numerical solution.
		\end{center}
	\end{minipage}
	\begin{minipage}{0.48\linewidth}
		\begin{center}	
			\includegraphics[width=6cm]{
				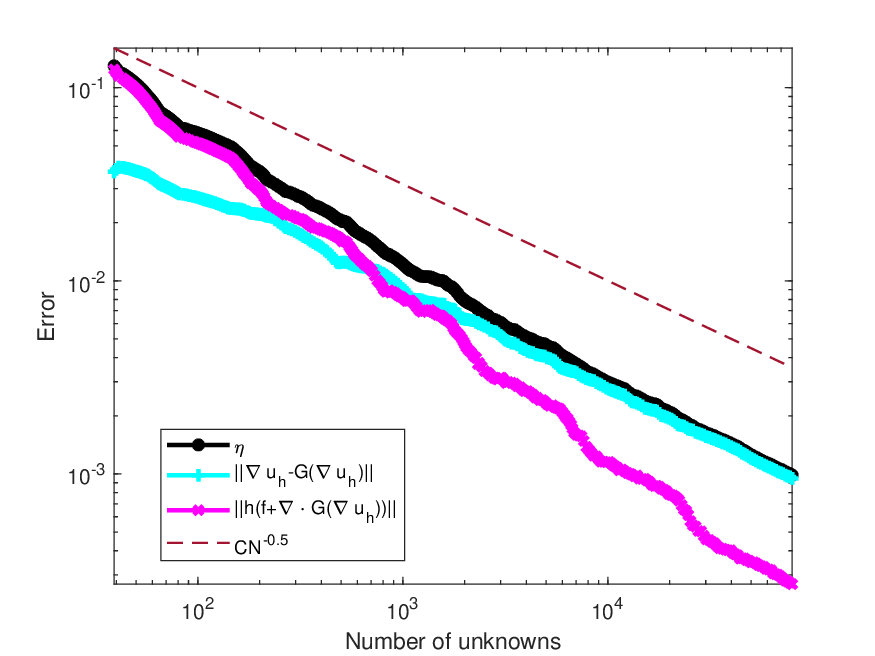}\\
			(d) History of the error estimator.
		\end{center}
	\end{minipage}
	\caption{Numerical results of Example \ref{test1}.}\label{test_1}
\end{figure}
\end{example}

\begin{example} \label{test3}
In this example, we consider the Poisson equation with the Dirichlet boundary condition given by the exact solution 
\[u=\text{atan}(S(\sqrt{(x-1.25)^2+(y+0.25)^2}-\pi/3)),\]
where $S=60$ reflects the steepness of the inner slope, and $f$ and $g$ are obtained from the exact solution $u$. 

Figure \ref{test_3} plots the numerical solution, locally refined mesh, errors of the gradient and error estimators. 
It is easy to see that the error estimator successfully guide the mesh refinement along the interface, where the solution have large gradient. 
Figure \ref{test_3} (c) reports the convergence history of errors and error estimator. Notice that the convergence rates of $\|\nabla u-\nabla u_h\|$ and error estimator $\eta$ are quasi-optimal, $\| \nabla u- G(\nabla u_h)\|$ and $\|h(f+\nabla\cdot G(\nabla u_h))\|$ are superconvergent. It can be seen clearly that the error estimator $\eta$ is asymptotically exact, and the term $\|h(f+\nabla\cdot G(\nabla u_h))\|$ is much smaller than the term $\|\nabla u_h-G(\nabla u_h)\|$.

\begin{figure}[!htbp]
	\begin{minipage}{0.48\linewidth}
		\begin{center}
			\includegraphics[width=7cm]{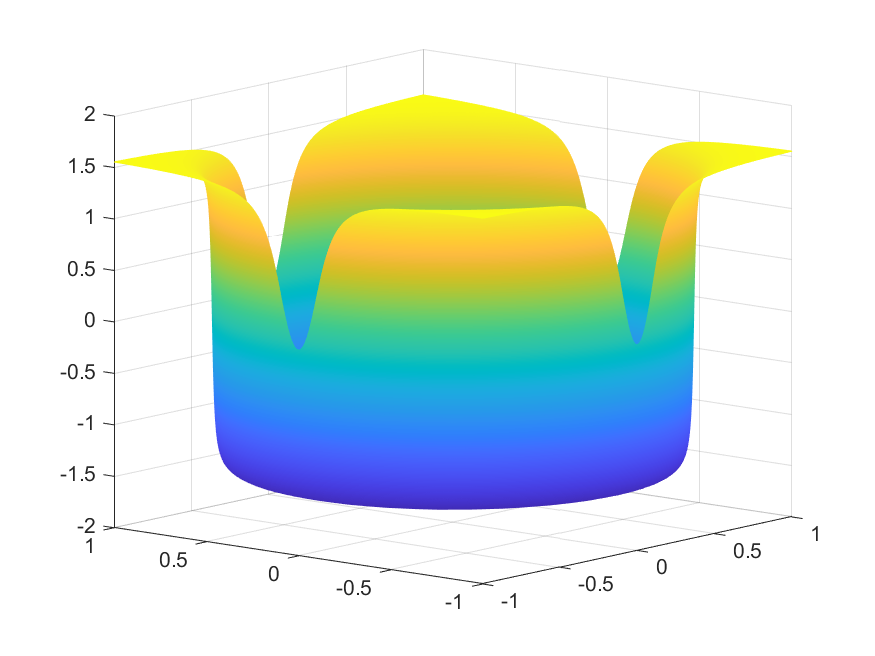}\\
			(a) Numerical solution.
		\end{center}
	\end{minipage}
	\begin{minipage}{0.48\linewidth}
		\begin{center}
			\includegraphics[width=7cm]{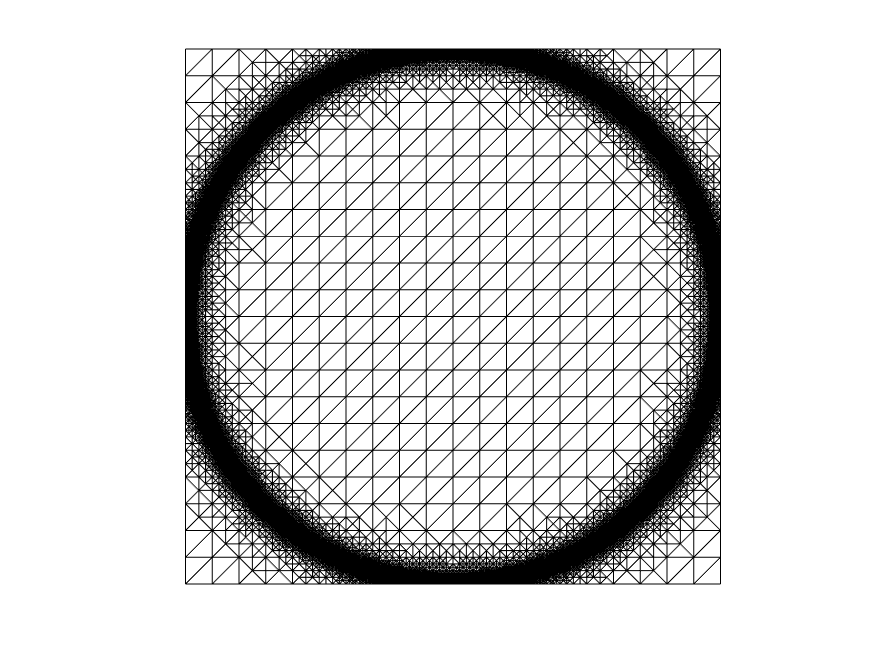}\\
			(b) Adaptive mesh.
		\end{center}
	\end{minipage}
	\begin{minipage}{0.48\linewidth}
		\begin{center}			
			\includegraphics[width=6cm]{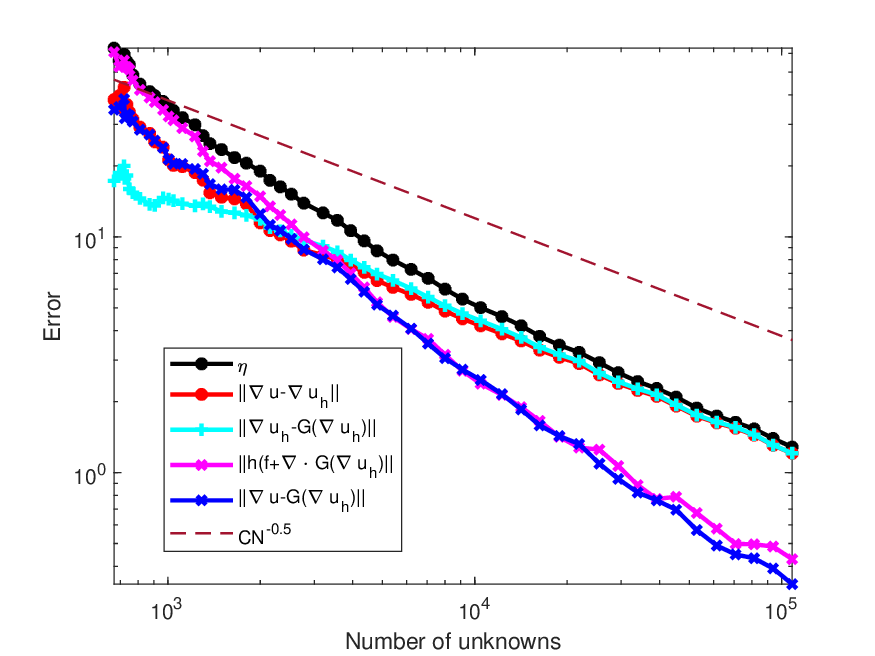}\\
			(c) History of error and error estimator.
		\end{center}
	\end{minipage}
	\caption{Numerical results of Example \ref{test3}.}\label{test_3}
\end{figure}

\end{example}

\begin{example} \label{test3d}
In this example, we consider the Poisson equation with the Dirichlet boundary condition on a 3D domain $\Omega=(-1,1)^3\backslash (0,1)\times(-1,0)$, the exact solution  
\[u=r^{2/3}\sin(2\theta/3),\]
$f$ and $g$ are obtained by the exact solution $u$.

\begin{figure}[!htbp]
	\begin{minipage}{0.48\linewidth}
		\begin{center}
			\includegraphics[width=7cm]{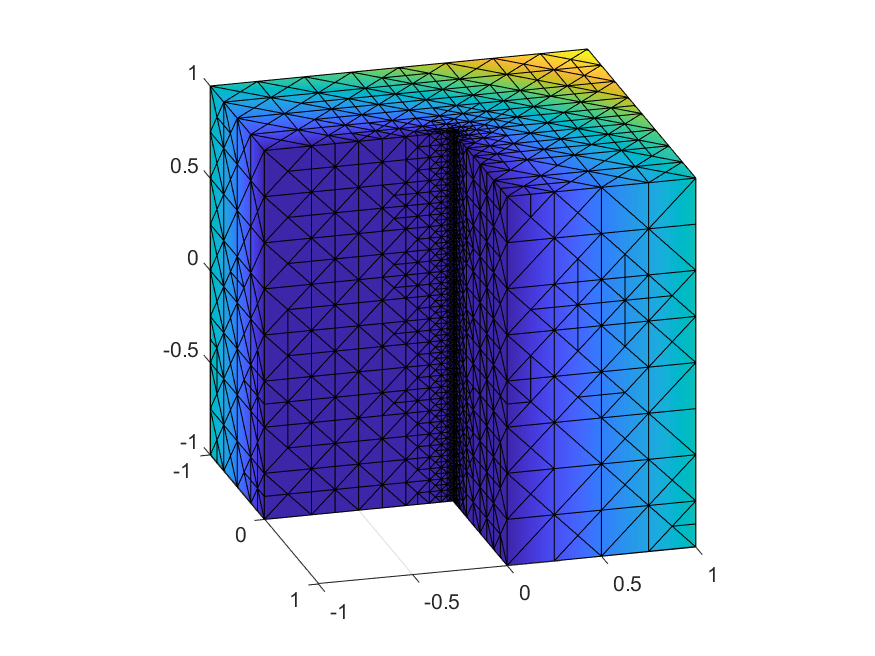}\\ 
			(a) Numerical solution.
		\end{center}
	\end{minipage}
	\begin{minipage}{0.48\linewidth}
		\begin{center}
			\includegraphics[width=7cm]{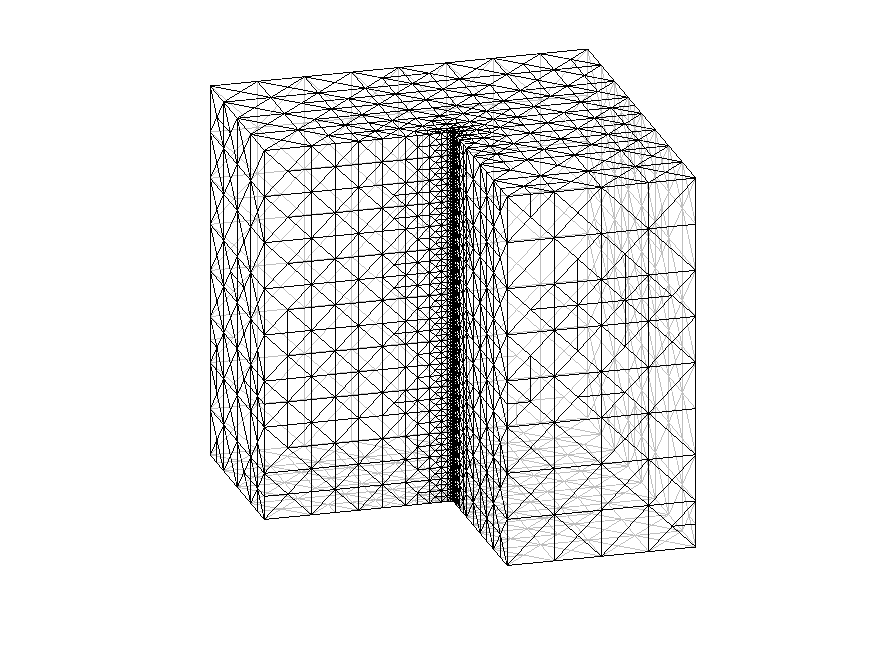}\\ 
			(b) Adaptive mesh.
		\end{center}
	\end{minipage} 
	\begin{minipage}{0.48\linewidth}
		\begin{center}			
			\includegraphics[width=6cm]{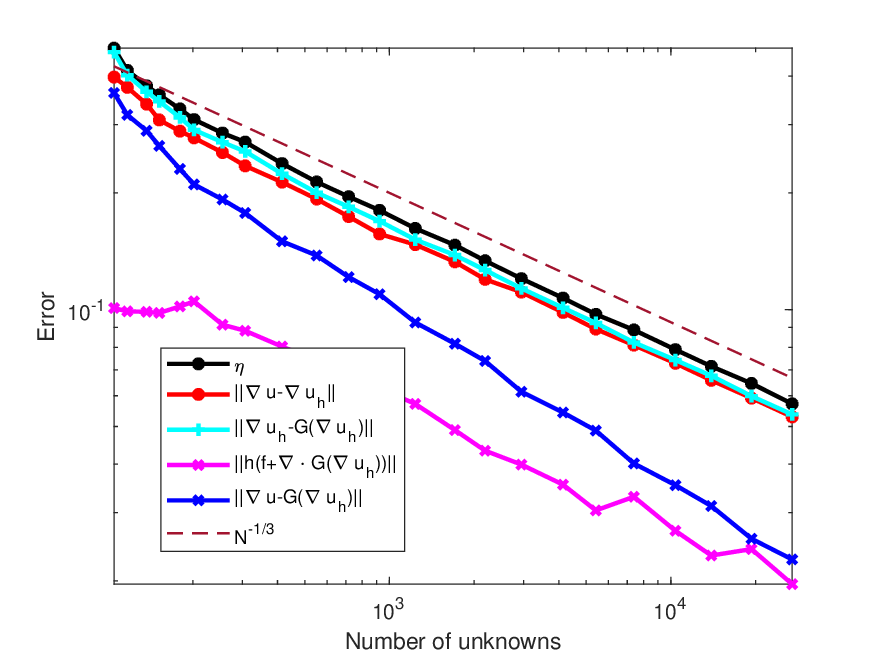}\\ 
			(c) History of error and error estimator.
		\end{center}
	\end{minipage}
	\caption{Numerical results of Example \ref{test3d}.}\label{test_3d}
\end{figure}

Figure \ref{test_3d} lists the numerical solution, adaptive mesh, and errors of Example \ref{test3d}. 
From Figure \ref{test_3d} (c), we can see clearly that gradient error $\|\nabla u-\nabla u_h\|$ and error estimator $\eta$ achieve quasi-optimal convergence rates, and the estimator is asymptotically exact. Similar to the previous three examples, in $\eta^2=\|G(\nabla u_h)-\nabla u_h\|^2+\|h(f+\nabla\cdot G(\nabla u_h))\|^2$, the part $\|h(f+\nabla\cdot G(\nabla u_h))\|$ is much smaller than the part $\|\nabla u_h-G(\nabla u_h)\|$.
\end{example}

In the following two examples, we extend our improved error estimator to the elliptic equation with diffusion coefficient $A$ with a modification, 
\[\eta_K^2=\|A^{1/2}(G(\nabla u_h)-\nabla u_h)\|^2_{0,K}+h_K^2\|A^{-1/2}({f}+\nabla\cdot(A G(\nabla u_h)))\|_{0,K}^2.\]

\begin{example} \label{test4}
Let $\Omega=[-1,1]\times[-1,1]$, we consider the model problem \eqref{Model problem} with continuous coefficient $A =\begin{pmatrix}
    10\cos(y) & 0\\ 0 & 10\cos(y)
\end{pmatrix}$ and the exact solution 
\[u=\frac{1}{(x+0.5)^2+(y-0.5)^2+0.01}-\frac{1}{(x-0.5)^2+(y+0.5)^2+0.01}.\]	

The numerical solution, adaptive mesh, and the errors are reported in Figure \ref{test_4}. The exact solution has large gradients near the two points $(-0.5, 0.5)$ and $(0.5, -0.5)$. 
Figure \ref{test_4} (b) shows that the error estimator guides the adaptive algorithm locally refined around the two peaks of the solution. 
In Figure \ref{test_4} (c), we can see that error estimator $\eta$ is asymptotically exact to the error $\|A^{1/2}(\nabla u-\nabla u_h)\|$. 
Moreover, the first part of $\eta$, i.e.  $\| A^{1/2}(\nabla u_h-G(\nabla u_h))\|$, dominates the error estimator $\eta$, and the second part of $\eta$, i.e. $\|A^{-1/2}h(f+\nabla \cdot(A G(\nabla u_h)))\|$, is superconvergent.
Both the mesh refinement and the errors demonstrate that the error estimator leads to a very effective convergent procedure.  

\begin{figure}[!htbp]
	\begin{minipage}{0.48\linewidth}
		\begin{center}
			\includegraphics[width=7cm]{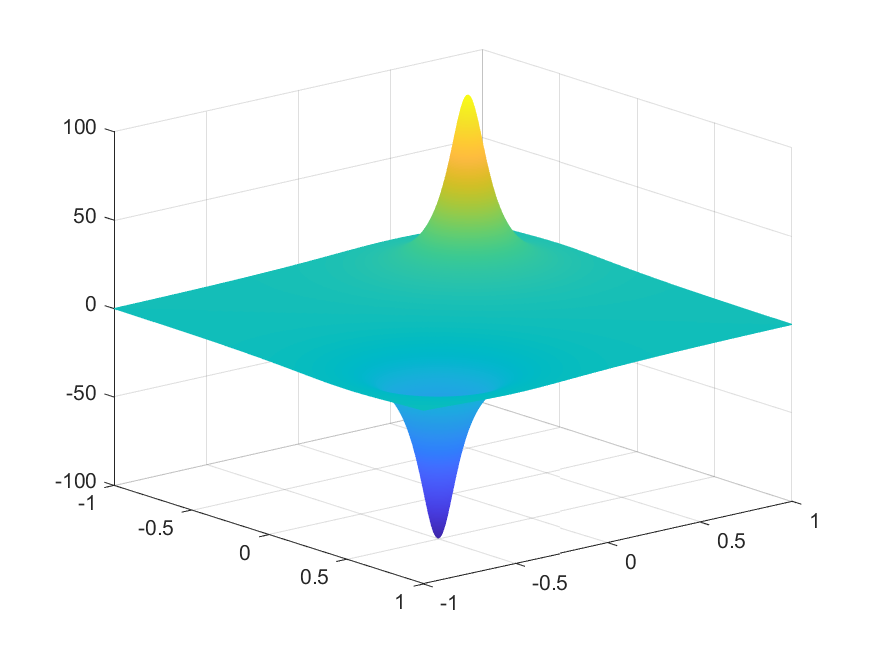}\\
			(a) Numerical solution.
		\end{center}
	\end{minipage}
	\begin{minipage}{0.48\linewidth}
		\begin{center}
			\includegraphics[width=7cm]{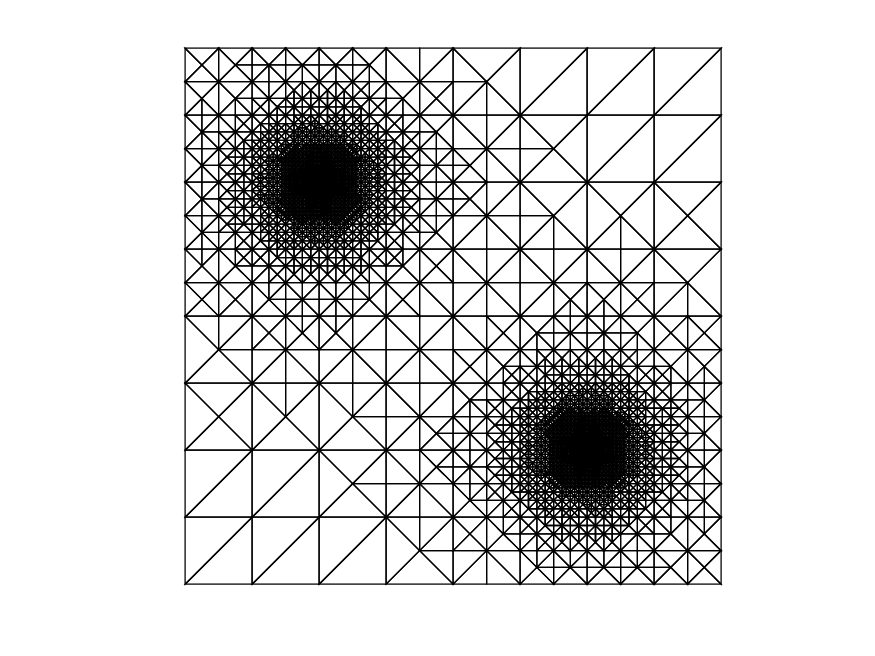}\\ 
			(b) Adaptive mesh.
		\end{center}
	\end{minipage}
	\begin{minipage}{0.48\linewidth}
		\begin{center}			
			\includegraphics[width=6cm]{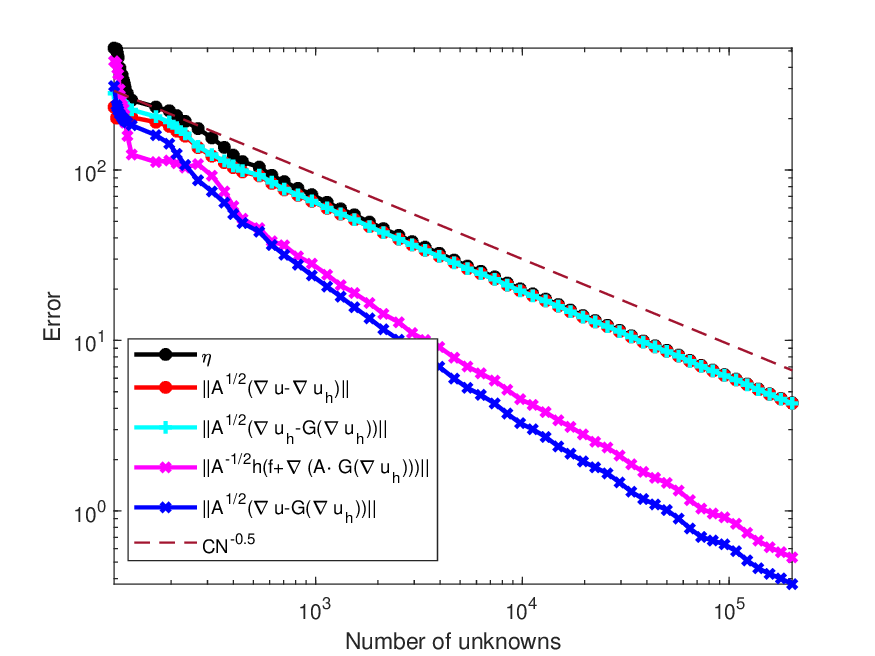}\\
			(c) History of error and error estimator.
		\end{center}
	\end{minipage}
	\caption{Numerical results of Example \ref{test4}.}\label{test_4}
\end{figure}

\end{example}
 
\begin{example} \label{test5} 
In this example, we consider the problem \eqref{Model problem} with discontinuous coefficient $A$ on a square domain $\Omega=(-1,1)\times (-1,1)$, which is decomposed into 4 sub-domain $\{\Omega_i\}_{i = 1}^4$, with $\Omega_1 = (0,1)\times(0,1),\,\Omega_2 = (-1,0)\times(0,1),\,\Omega_3 = (-1,0)\times(-1,0)$ and $\Omega_4 = (0,1)\times(-1,0)$. The exact solution is
\[u(r,\theta)=r^\gamma\mu(\theta),\]
with
	\[
	\mu(\theta)=\begin{cases}\cos((\pi/2-\sigma)\gamma\cdot \cos((\theta-\pi/2+\rho)\gamma)&if~\, 0\leq\theta\leq\pi/2,\\
		\cos(\rho\gamma)\cdot \cos((\theta-\pi+\sigma)\gamma)&if~\, \pi/2\leq\theta\leq\pi,\\
		\cos(\sigma\gamma)\cdot \cos((\theta-\pi-\rho)\gamma)&if~\,\pi\leq\theta<3\pi/2,\\
		\cos((\pi/2-\rho)\gamma)\cdot \cos((\theta-3\pi/2-\sigma)\gamma)&if ~\,3\pi/2\leq \theta<2\pi,
	\end{cases}\]
	where $\gamma, \rho, \sigma$ are numbers satisfying the following nonlinear relations
\[\begin{cases}
		R:=a_1/a_2=-\tan((\pi/2-\sigma)\gamma)\cdot \cot(\rho\gamma),\\
		1/R=-\tan(\rho\gamma)\cdot \cot(\sigma\gamma),\\
		R=-\tan(\sigma\gamma)\cdot\cot((\pi/2-\rho)\gamma),\\
		0<\gamma<2,\\
		\max\{0,\pi\gamma-\pi\}<2\gamma\rho<\min\{\pi\gamma,\pi\},\\
		\max\{0,\pi-\pi\gamma\}<-2\gamma\sigma<\min\{\pi,2\pi-\pi\gamma\}.
	\end{cases}\]
	The piecewise constant coefficient A is taken as
\[A =\begin{cases}
		a_1I&~~xy>0,\\
		a_2I&~~xy<0.
	\end{cases}\]
	In this example, the numbers $\gamma, \rho, \sigma$ are chosen as
\begin{center}
\begin{tabular}{ccccc}
		\hline
		$\gamma$&$\sigma$&$\rho$&$a_1$&$a_2$\\
		\hline
$0.1$&$-14.92256510455152$&$\pi/4$&$161.4476387975881$&$1.0$
  \\	\hline
	\end{tabular}
 \end{center}

The solution has a discontinuous derivative along the interface, and the numerical solution is plotted in Figure \ref{test_50} (a).
Note that the gradient of the solution is discontinuity across the interface, then we apply the gradient recovery method on each sub-domain $\{\Omega_i\}_{i = 1}^4$ separately.
Adaptive refined mesh and the convergence history of errors and the estimators are show in Figure \ref{test_50} (b)-(c). 
We see that the gradient recovery-based error estimator successfully guide the mesh refinement around the origin point, and the decay of $\|A^{1/2}(\nabla u-\nabla u_h)\|$ and $\eta$ is quasi-optimal. 

\begin{figure}[!htbp]
	\begin{minipage}{0.48\linewidth}
		\begin{center}
			\includegraphics[width=7cm]{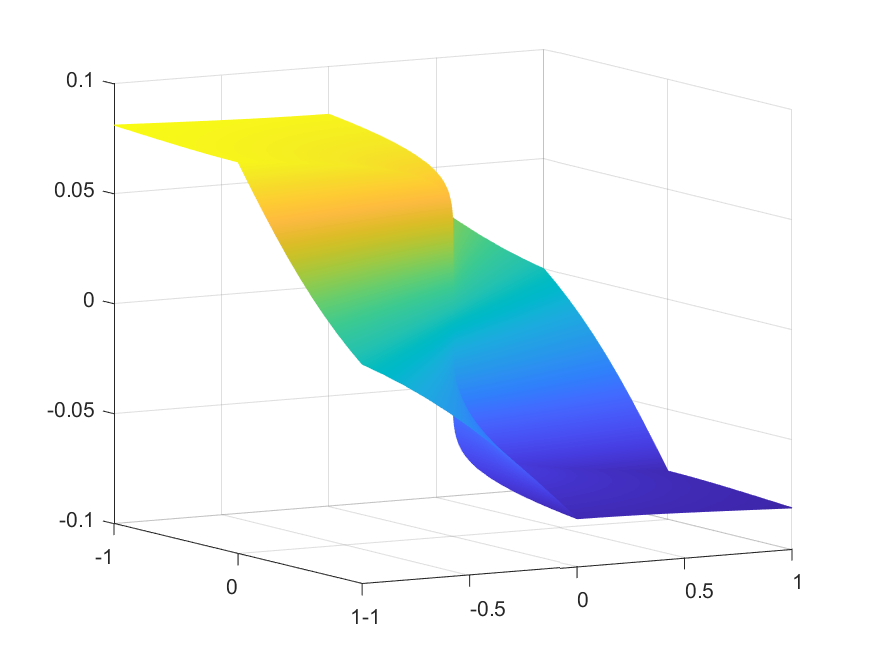}\\ 
			(a) Numerical solution.
		\end{center}
	\end{minipage}
 	\begin{minipage}{0.48\linewidth}
		\begin{center}
			\includegraphics[width=7cm]{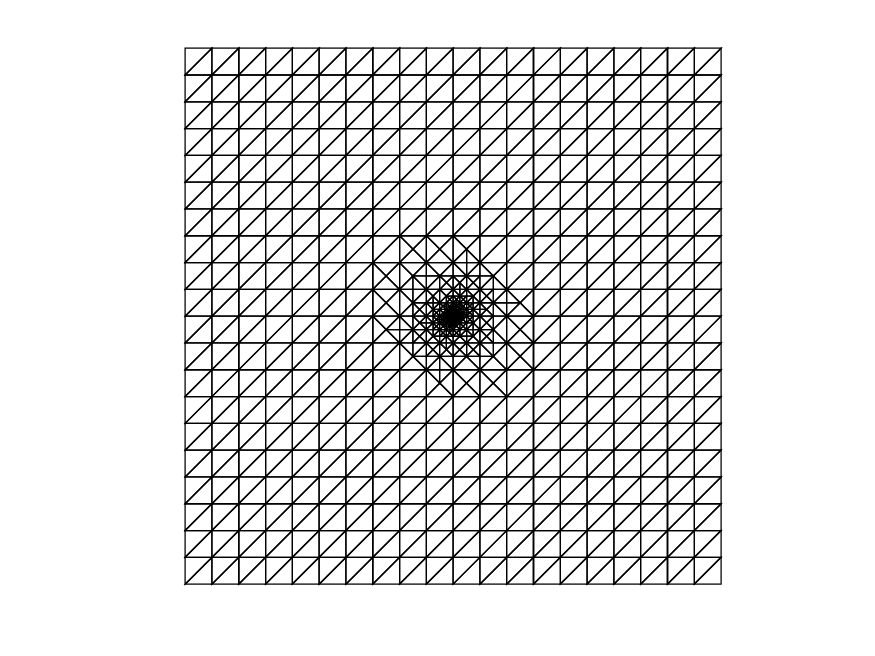}\\ 
			(b) Adaptive mesh.
		\end{center}
	\end{minipage}

	\begin{minipage}{0.48\linewidth}
		\begin{center}			
			\includegraphics[width=6cm]{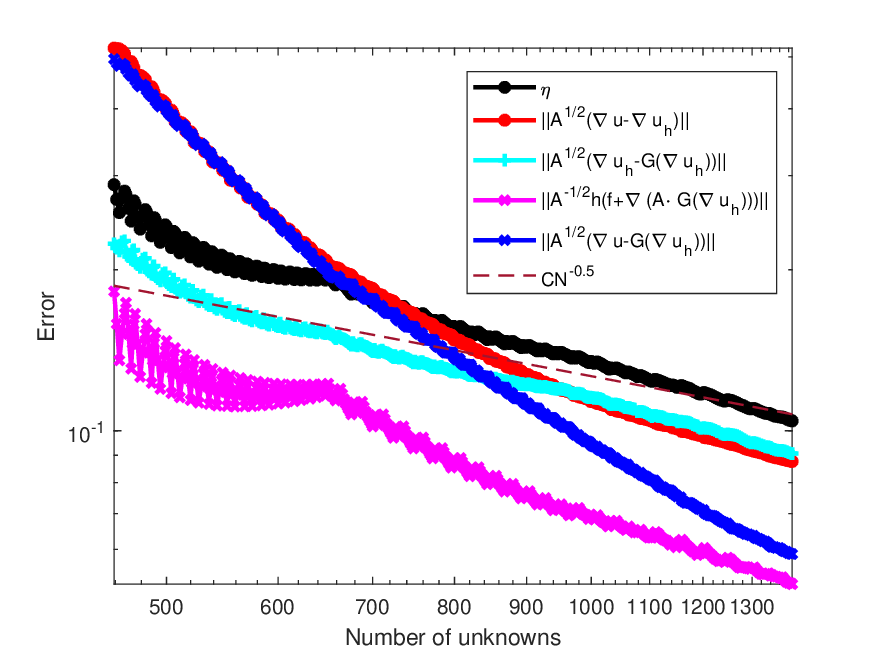}\\ 
			(c) History of error and error estimator.
		\end{center}
	\end{minipage}
 \caption{Numerical results of Example \ref{test5}.}\label{test_50}
 \end{figure}
 
\end{example}

\section*{Acknowledgments}
 Y. Liu was supported by the NSFC Project (12301473), the Doctoral Starting Foundation of Xi’an University of Technology, China
(109-451123001) and Natural Science Special Project of Shaanxi Provincial Department of Education (23JK0564). 
Yi's research was partially supported by NSFC Project (12431014), 
Project of Scientiﬁc Research Fund of the Hunan Provincial Science and Technology Department (2024ZL5017), 
and Program for Science and Technology Innovative Research Team in Higher Educational Institutions of Hunan Province of China.

%\bibliographystyle{unsrt}
%\bibliographystyle{plain}
%\bibliography{referencefile}

\end{document}